\def\year{2020}\relax
\documentclass[letterpaper]{article} 
\usepackage{aaai20}  
\usepackage{times}  
\usepackage{helvet} 
\usepackage{courier}  
\usepackage[hyphens]{url}  
\usepackage{graphicx} 
\usepackage{graphicx}  
\usepackage{amssymb, amsmath, amsthm, latexsym}
\usepackage[algo2e]{algorithm2e} 
\usepackage{algorithm}
\usepackage{tabularx}
\usepackage{booktabs}       
\usepackage{color}
\usepackage{pifont}

\newcommand{\R}{\mathbb{R}}

\newtheorem{theorem}{Theorem}
\newtheorem{thm}{Theorem}
\newtheorem{lemma}{Lemma}
\newtheorem{assumption}{Assumption}
\newcommand{\eqdef}{\stackrel{\text{def}}{=}}

\newcommand{\circledOne}{\text{\ding{172}}\;}
\newcommand{\circledTwo}{\text{\ding{173}}\;}

\newcommand{\ve}[2]{\left\langle #1 , #2 \right\rangle}
\def\<#1,#2>{\left\langle #1,#2\right\rangle}

\title{A Stochastic Derivative-Free Optimization Method with Importance Sampling: Theory and Learning to Control}

\author{Adel Bibi\textsuperscript{\rm 1}, El Houcine Bergou\textsuperscript{\rm 1,2}, Ozan Sener\textsuperscript{\rm 3}, Bernard Ghanem\textsuperscript{\rm 1} and Peter Richt\'arik\textsuperscript{\rm 1,4} \\
\textsuperscript{\rm 1}King Abdullah University of Science and Technology\\
\textsuperscript{\rm 2} MaIAGE, INRA, Universit\'e Paris-Saclay\\
\textsuperscript{\rm 3}Intel Labs\\
\textsuperscript{\rm 4}Moscow Institute of Physics and Technology\\
\{adel.bibi,houcine.bergou,bernard.ghanem,peter.richtarik\}@kaust.edu.sa, ozan.sener@intel.com}

\begin{document}

\maketitle

\begin{abstract} 
    We consider the problem of unconstrained minimization of a smooth objective function in $\R^n$ in a setting where only function evaluations are possible. While importance sampling is one of the most popular techniques used by machine learning practitioners to accelerate the convergence of their models when applicable, there is not much existing theory for this acceleration in the derivative-free setting. In this paper, we propose the first derivative free optimization method with importance sampling and derive new improved complexity results on non-convex, convex and strongly convex functions. We conduct extensive experiments on various synthetic and real LIBSVM datasets confirming our theoretical results. We test our method on a collection of continuous control tasks on MuJoCo environments with varying difficulty. Experiments show that our algorithm is practical for high dimensional continuous control problems where importance sampling results in a significant sample complexity improvement.
\end{abstract}
\section{Introduction}
In this paper, we consider the optimization problem
\begin{equation}\label{eq:mainP}
\min_{x\in \R^n} f(x),
\end{equation}
where $f: \R^n \rightarrow \R$ is a ``smooth'' but not necessarily convex  function, bounded from below and it achieves its global minimum at some $x_*\in \R^n$. In particular, we enforce throughout the paper the following smoothness assumption: \begin{assumption}\label{ass:M-smooth} 
The objective function $f$ has coordinate-wise Lipschitz gradient, with Lispchitz constants $L_1,\dots,L_n>0$. Moreover, $f$ is bounded from below by $f(x_*)\in \R$. That is, $f$ satisfies
\begin{align*}   
f(x_*) \leq f(x+t e_i) \le f(x) + \nabla_i f(x) t  + \frac{L_i }{2}t^2 
\end{align*}
for all $x \in \R^n$ and $t\in \R$, where $\nabla_i f(x)$ is the $i$th partial derivative of $f$ at $x$.
\end{assumption} 

\noindent \textbf{DFO.}
We consider the Derivative-Free Optimization (DFO) \cite{Conn_2009,Kolda_2003} setting. That is, we assume that the derivatives of $f$ are numerically impractical to obtain, unreliable (e.g., noisy {\em function} evaluations \cite{chen2015stochastic}), or not available at all. In typical DFO applications, evaluations of $f$ are possible through runs/simulations of some black-box software only. Optimization problems of this type appear in many applications, including computational medicine \cite{Marsden_2008}, fluid-dynamics \cite{Allaire_2001,Haslinger_2003}, localization \cite{Marsden_2004,Marsden_2007} and continuous control \cite{Mania_2018,Salimans_2017}.

 Literature on DFO methods for solving (\ref{eq:mainP}) has a long history. Some of the first approaches were based on deterministic direct search (DDS) \cite{Hooke_1961}. Subsequently, additional variants of DDS, including randomized  approaches, were proposed in \cite{Matyas_1965,Karmanov_1974a,Karmanov_1974b,Baba_1981,Dorea_1983,Sarma_1990}. However,  complexity bounds for deterministic direct search methods have only been established recently by the works of \cite{Vicente_2013,garmanjani2013smoothing,KR-DFO2014,Vicente_2016}.
 Recently, complexity bounds have also been derived for randomized methods \cite{Diniz_2008,Stich_2011,Gratton_2015}. For instance, the work of \cite{Diniz_2008,Gratton_2015} imposes a decrease condition on whether to accept or reject a step of a set of random directions. Moreover, \cite{Nesterov_Spokoiny_2017,Dvurechensky_Gasnikov_Gorbunov_2018} derived new complexity bounds for accelerated random search. 
 
More recently, Bergou et. al. proposed a new randomized direct search method called {\em Stochastic Three Points} (\texttt{STP}) method. \texttt{STP}, in each iteration $k$, generates a random search direction $s_k$ according to a certain probability law, then compares the objective function at three points: current iterate $x_k$, a  point $x_+ = x_k+\alpha_k s_k$ in the direction of $s_k$ and a point $x_- = x_k - \alpha_k s_k$ in the direction of $-s_k$. The method then chooses the best of these three points as the new iterate:
 \[x_{k+1} = \arg\min \{f(x_k),  f(x_+), f(x_-)\}.\]

 {\bf Notation: }
As for the notations, $\mathbb{E} \left[\cdot \right]$ denotes the expectation operator. The standard inner product is defined as $\langle x,y \rangle = x^\top y$. We also denote the $\ell_1$-norm and $\ell_2$-norm by $\| \cdot \|_1$ and $\| \cdot \|_2$, respectively. We define $L = \max_i L_{i}$ for a given sequence of scalars $L_1,\ldots,L_n$.

\begin{algorithm}[t]
\caption{{\bf Stochastic  Three Points Method with Importance Sampling}  (\texttt{STP}$_{\text{\texttt{IS}}}$)}
\label{alg:STP_IS}
\begin{rm}
\begin{description}
\item[]
\item[Initialization] \ \\
Choose initial iterate $x_0\in \R^n$, stepsize parameters $v_1,\dots,v_n > 0$ and probabilities  $p_1,\dots,p_n >0$ summing up to 1. 
\item[For $k=0,1,2,\ldots$] \ \\
\vspace{-2ex}
\begin{enumerate}
\item Select $i_k=i$ with probability $p_i>0$. 
\item Choose stepsize $\alpha_{i_k}$ proportional to $1/v_{i_k}$.
\item Let $x_+ = x_k+\alpha_k e_{i_k}$ and $x_- = x_k - \alpha_k e_{i_k}$
\item  $x_{k+1} = \arg \min \{f(x_k), f(x_+), f(x_-)\}$
\end{enumerate}
\end{description}
\end{rm}
\end{algorithm}

\section{Paper Overview and Contributions}
While importance sampling, a term that typically refers to the nonuniform sampling of random directions in stochastic algorithms, has been widely investigated in gradient based methods \cite{zhao2015stochastic,qu2015quartz,richtarik2016optimal,NIPS2017_7025}, to the best of our knowledge {\em there exists no work on importance sampling in the random direct search setting.} To this end, we study \texttt{STP} and analyze its complexity with arbitrary probabilities. In particular, we restrict the random directions to be sampled from discrete distributions, i.e., in each iteration of \texttt{STP} a random direction $s_k$ from a finite set of independent directions $\{b_1,\dots,b_n\}\subset \R^n$ is sampled. That is, we set $s_k=b_i$ with  probability $p_i>0$. We then propose new sampling strategies that are either optimal or at least improve the complexity bounds, i.e., {\em importance sampling}.

\subsection{Coordinate directions}
\label{sec:coordinate_directions}
Without loss of generality, we  only consider  directions in the canonical basis of $\R^n$, i.e., $e_1,\dots,e_n$. The general case can be recovered via a linear change of variables: $x=By$, where $B\in \R^{n\times n}$. Indeed, consider the problem
\begin{equation}\label{eq:hg89f8g9f}\min_{y\in \R^n} f_B(y) \eqdef f(By)\end{equation}
instead. A coordinate update $y_{k+1} = y_k + \alpha_k e_i$ for the re-parameterized problem \eqref{eq:hg89f8g9f} corresponds to updates of the form $x_{k+1} = x_k + \alpha_k b_{i}$, where $b_i$ is the $i$th column of $B$, for the original problem \eqref{eq:mainP}. In light of the above discussion, the newly proposed algorithm dubbed \texttt{STP}$_{\text{\texttt{IS}}}$ is formally described as Algorithm \ref{alg:STP_IS}.

\subsection{Complexity bounds}

To the best of our knowledge, ours are the {\em first complexity bounds (bounds on the number of iterations) for a DFO method with importance sampling}. We design importance sampling that improves the worst-case iteration complexity bounds compared to state-of-the-art algorithms. These bounds have the same dependence on the precision $\epsilon$ as classical bounds in the literature,  i.e.  $1/\epsilon^2$ for non-convex $f$, $1/\epsilon$ for convex $f$ and $\log(1/\epsilon)$ for strongly convex $f$; see for instance   \cite{Bergou_2018,Nesterov_Spokoiny_2017}. However, the leading  constant, which is often the bottleneck in practical performance, especially when low or medium accuracy solutions are acceptable, is improved and often dramatically so. Typically, the improvement is via replacing the maximum Lipschitz constant of the gradient by the average Lipschitz constants of all coordinates (see Theorems \ref{thm:nonconvex1}, \ref{thm:nonconvex2}, \ref{thm:convex2}, and \ref{thm:stronglyconvex2}). The improvement we obtain is similar to the improvement obtained by importance sampling in stochastic coordinate (gradient) descent methods~ \cite{zhao2015stochastic,qu2015quartz,richtarik2016optimal}. Table~\ref{tab:sumcompl} summarizes complexity results obtained in this paper for \texttt{STP} and for \texttt{STP}$_{\text{\texttt{IS}}}$. The assumptions in Table~\ref{tab:sumcompl} are in addition to Assumption \ref{ass:M-smooth}.

\subsection{Empirical results} In addition to our theoretical analysis, we conduct extensive testing to show the efficiency of the proposed method in practice. We use both synthetic and real\footnote{We use several LIBSVM datasets~\cite{chang2011libsvm}.} datasets for ridge regression and squared SVM problems. In the non-convex case, we use continuous control tasks from the MuJoCo~\cite{Todorov_2012} suite following the recent success of DFO  compared to model-free RL~\cite{Mania_2018,Salimans_2017}. Results show that {\em our approach leads to huge speedups compared against uniform sampling, the improvement can reach several orders of magnitude and comparable or better than state-of-art policy gradient methods.}

\begin{table*}[t]
\centering
\renewcommand{\arraystretch}{1.25}{
\begin{tabular}{c|c|c|c|c}
\toprule
Assumptions on $f$    & \begin{tabular}{@{}c@{}}Uniform Sampling \\ Complexity\end{tabular}  & \begin{tabular}{@{}c@{}}Importance\\ Sampling\end{tabular} & \begin{tabular}{@{}c@{}}Importance Sampling \\Complexity [NEW]\end{tabular} & Theorem  \\ 
  \midrule
None & $\frac{4\sqrt{2} r_0{n^2L}}{\epsilon^2}$ & $p_i =  \frac{\sqrt{L_i}}{\sum_{i=1}^n \sqrt{L_i}}$ & $\frac{4\sqrt{2} r_0{\left(\sum_{i=1}^n \sqrt{L_i}\right)^2}}{\epsilon^2}$ 
& \ref{thm:nonconvex1} \\

None & $\frac{4\sqrt{2} r_0{n^2L}}{\epsilon^2}$ & $p_i = \frac{L_i}{\sum_{i=1}^n L_i} $ & $ \frac{4\sqrt{2} r_0{n\left(\sum_{i=1}^n L_i\right)}}{\epsilon^2} $ & \ref{thm:nonconvex2} \\

Convex, $R_0 < \infty$  & $8R_0^2{n^2L}\left(\frac{1}{\epsilon}- \frac{1}{r_0}\right)$ & $p_i = \frac{L_i}{\sum_{i=1}^n L_i} $ & $8R_0^2{n\sum_{i=1}^n L_i}\left(\frac{1}{\epsilon}- \frac{1}{r_0}\right)$ & \ref{thm:convex2} \\

$\lambda$-strongly convex  & $\frac{{nL}}{\lambda}\log\left(\frac{r_0}{\epsilon}\right)$ & $p_i =  \frac{L_i}{\sum_{i=1}^n L_i} $ & $\frac{{\sum_{i=1}^n L_i}}{\lambda}\log\left(\frac{r_0}{\epsilon}\right)$ & \ref{thm:stronglyconvex2}
\\
\bottomrule
\end{tabular}
}
\caption{Summary of the new derived complexity results as opposed to uniform sampling where $r_0 = f(x_0) - f(x_*)$. The assumptions listed are in addition to Assumption \ref{ass:M-smooth}. $R_0 < \infty$ indicates a bounded level set where the exact definition is given in Assumption \ref{ass:level_sets}. The key differences in complexity between the uniform and importance sampling are detailed in text.}
\label{tab:sumcompl}
\end{table*} 
\section{Non-Convex Case}\label{sec:nonconv}

This section describes our complexity results for Algorithm~\ref{alg:STP_IS} in the case when $f$ is allowed to be non-convex. We show that this method guarantees complexity bounds with the same order in $\epsilon$ as classical bounds in the literature, i.e., $1/\epsilon^2$ with an improved dependence on the Lipschitz constant. All proofs are left for the \textit{appendix}.

\begin{theorem}\label{thm:nonconvex1}
Let Assumption~\ref{ass:M-smooth} be satisfied. Choose $\alpha_{i_k}=\tfrac{\alpha_0}{v_{i_k} \sqrt{k+1}}$ where $\alpha_0>0$. If 
\begin{equation}
\label{eq:isjss8sus} 
K\geq \frac{2 \left( \frac{\sqrt{2}(f(x_0)-f(x_*))}{\alpha_0} + \frac{\alpha_0 \sum_{i=1}^n  \frac{p_i {L_i}}{{v_{i}^2}}}{2} \right)^2}{\left({\min_i \frac{p_i}{v_i}}\right)^2 \epsilon^2},
\end{equation}
then  $\displaystyle \min_{k=0,1,\dots,K} \mathbb{E} \left[ \|\nabla f(x_k)\|_{1} \right] \leq \epsilon.$
\end{theorem}
 Note that the complexity depends on $\alpha_0$. The optimal choice of $\alpha_0$ minimizing \eqref{eq:isjss8sus} is $\alpha_0^* = 8^{1/4}\sqrt{\frac{{f(x_0)-f(x_*)}}{\sum_{i=1}^n  \frac{p_i {L_i}}{{v_{i}^2}}}},$ in which case the complexity bound \eqref{eq:isjss8sus} takes the form
\begin{equation*} \frac{4\sqrt{2}\left(f(x_0)-f(x_*)\right) \sum_{i=1}^n  \frac{p_i {L_i}}{{v_{i}^2}}}{\left( \min_i \frac{p_i}{v_i}\right)^2 \epsilon^2}.\end{equation*}

{\bf Importance sampling.} The complexity depends on the choice of the probabilities $\left\{p_i\right\}_{i=1}^n$ and the quantities $\left\{v_i\right\}_{i=1}^n$. For instance, if $p_i = \frac{\sqrt{L_i}}{\sum_{i=1}^n \sqrt{L_i}}$
 and $v_i = \sqrt{L_i}$, then the complexity becomes
 \begin{equation}\label{eq:isjss8sus-optimal-1} \frac{4\sqrt{2}(f(x_0)-f(x_*)) \left(\sum_{i=1}^n \sqrt{L_i} \right)^2}{ \epsilon^2}.
 \end{equation}
 On the other hand, if $p_i = \frac{{L_i}}{\sum_{i=1}^n {L_i}}$
 and $v_i = {L_i}$ then the complexity becomes
 \begin{equation}\label{eq:isjss8sus-optimal-2} \frac{4\sqrt{2}(f(x_0)-f(x_*)) n \left(\sum_{i=1}^n {L_i} \right)}{ \epsilon^2}.
 \end{equation} 

Under the choice of uniform sampling, i.e. $p_i = \frac{1}{n}$ and the choice $v_i = L$, we recover the uniform sampling complexity of \cite{Bergou_2018} $$\frac{4\sqrt{2}(f(x_0)-f(x_*))n^2 L}{ \epsilon^2}.$$
Note that $ \left(\sum_{i=1}^n \sqrt{L_i}\right)^2 \le n^2 L$ and $n \left(\sum_{i=1}^n {L_i} \right) \le n^2 L$. Therefore, complexity bounds in the number of iterations is improved with the proposed importance sampling strategies (\ref{eq:isjss8sus-optimal-1}) and (\ref{eq:isjss8sus-optimal-2}). We now state a complexity theorem for Algorithm~\ref{alg:STP_IS} when using non-decreasing stepsizes.

\begin{theorem}\label{thm:nonconvex2} 
Let Assumption~\ref{ass:M-smooth} be satisfied. Choose $\alpha_{i_k}= \frac{\epsilon}{n v_{i_k}}$ where $\sum_{i=1}^n \frac{p_i L_i}{v_i^2} < 2n \left(\min_i \frac{p_i}{v_i}\right)$. If
\begin{equation}
\label{eq:complk}
   K \ge   \frac{2 n (f(x_0)-f(x_*))}{\left(\min_i \frac{p_i}{v_i} \right) \left(1 -\frac{\sum_{i=1}^n  \frac{p_i {L_i}}{{v_{i}^2}}}{2n\left(\min_i \frac{p_i}{v_i}\right)} \right)  \epsilon^2},
\end{equation}
then $\displaystyle \min_{k=0,1,\dots,K} \mathbb{E} \left[ \|\nabla f(x_k)\|_{1} \right] \leq \epsilon.$
\end{theorem}

Under the choice of importance sampling $p_i = \frac{L_i}{\sum_{i=1}^n {L_i}}$ and $v_i = {L_i}$, the complexity \eqref{eq:complk} becomes
\begin{equation}\label{eq:isjss8sus-optimal} 
 \frac{4 \left(f(x_0)-f(x_*)\right) n\left(\sum_{i=1}^n {L_i} \right)}{ \epsilon^2}.
\end{equation} 
Similar to Theorem \ref{thm:nonconvex1}, the uniform sampling complexity of \cite{Bergou_2018} can be recovered with $p_i = \frac{1}{n}$ and $v_i = L$. Note that the uniform sampling complexity is proportional to $n^2 L$ and since $n \sum_{i=1}^n {L_i} \le n^2 L$, the worst case complexity of the number of iterations is also improved for this choice of importance sampling.
\begin{figure*}[t]
\centering
\begin{tabular}{@{}c@{}c@{}c@{}}
\includegraphics[width=0.33\textwidth]{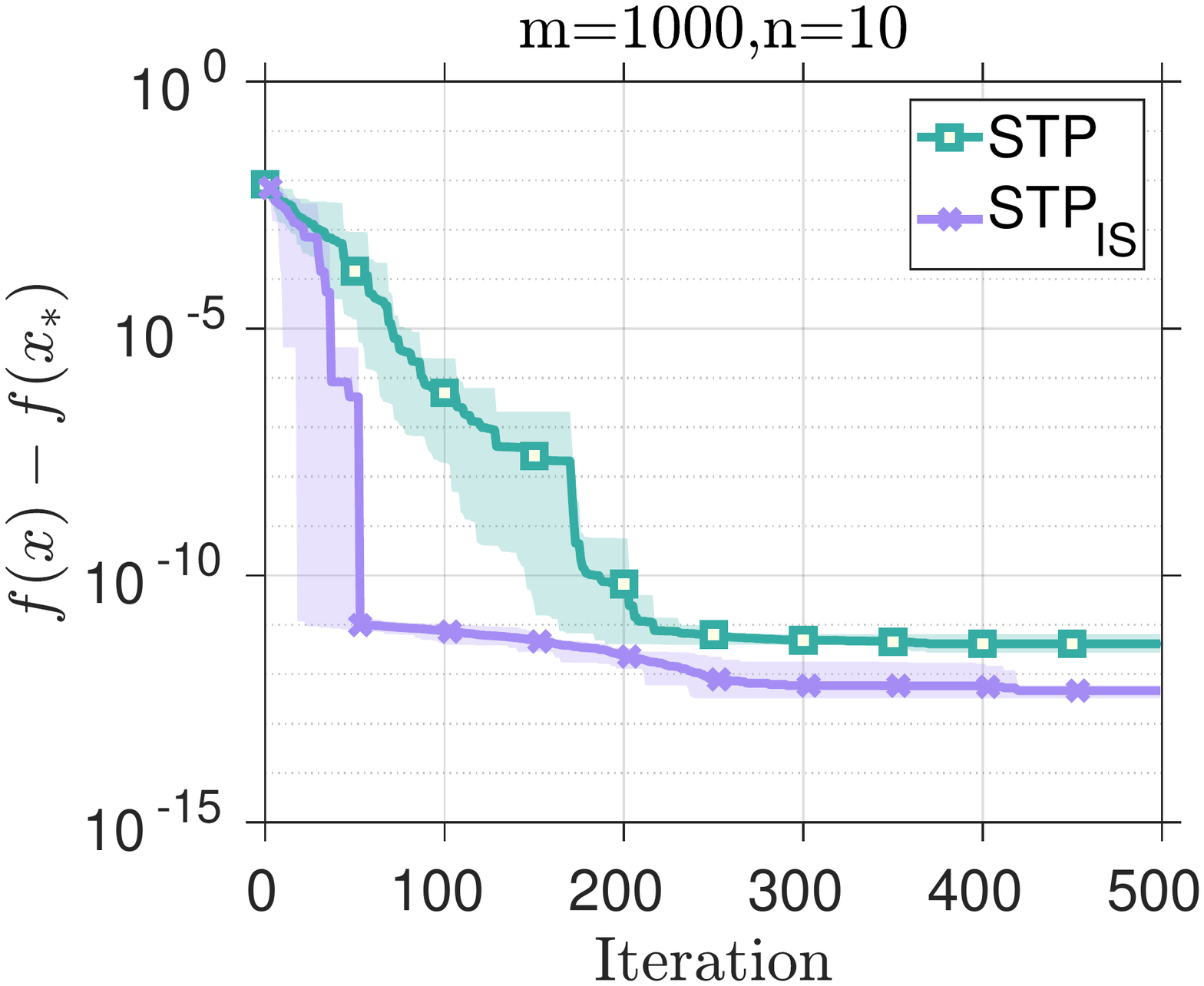}&
\includegraphics[width=0.33\textwidth]{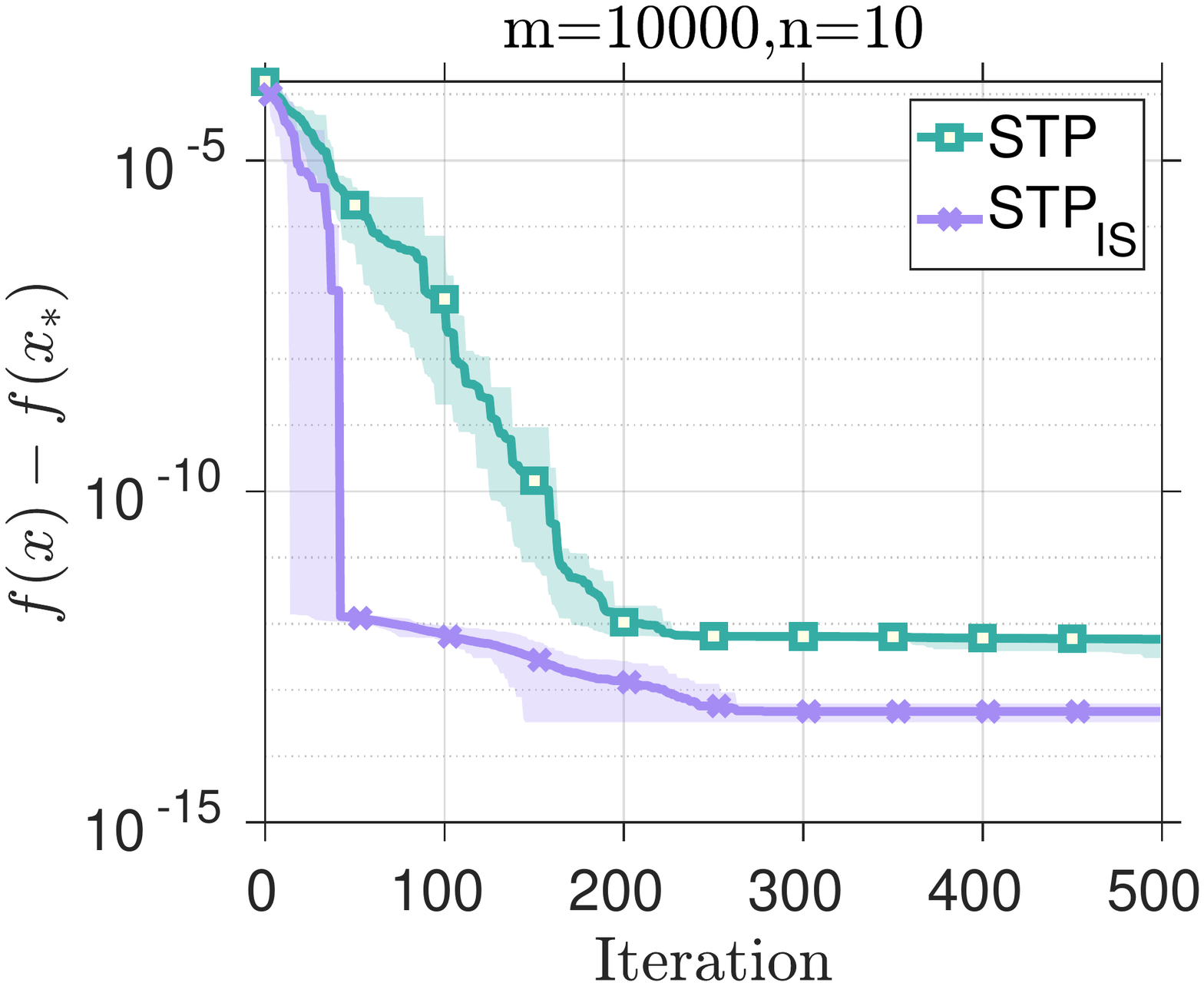}&
\includegraphics[width=0.33\textwidth]{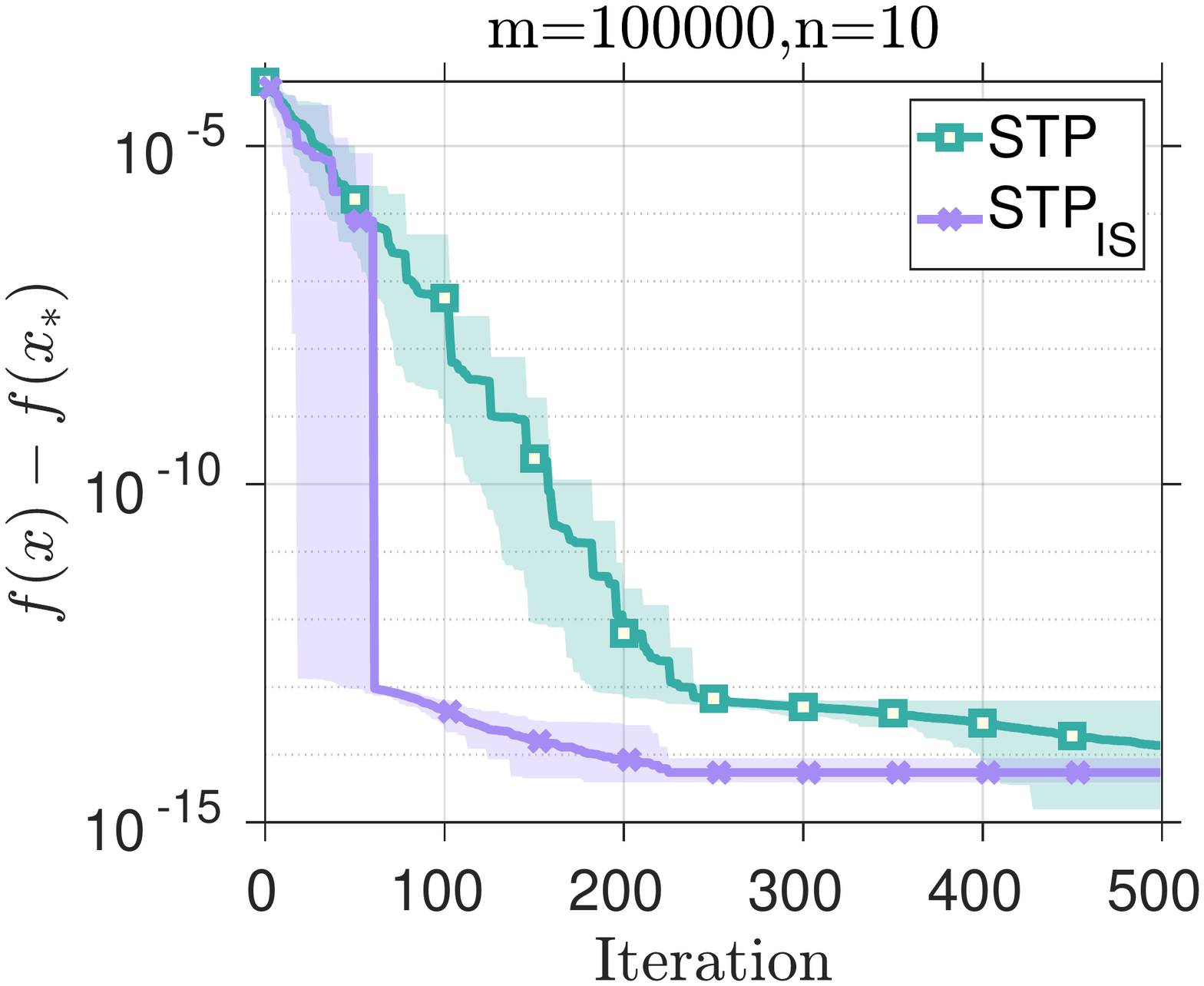}
\\ \vspace{-35mm} \\
\includegraphics[width=0.33\textwidth]{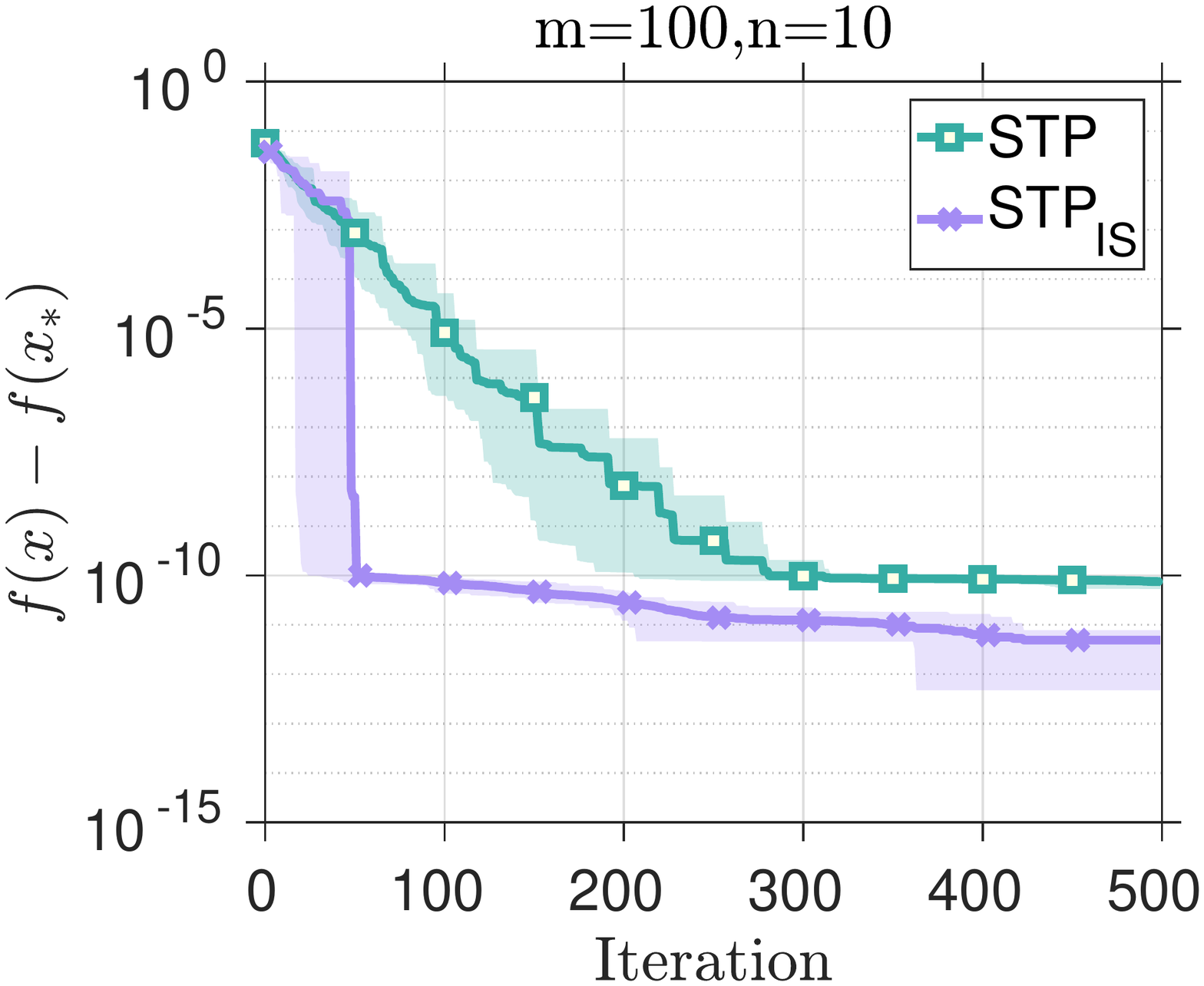} &
\includegraphics[width=0.33\textwidth]{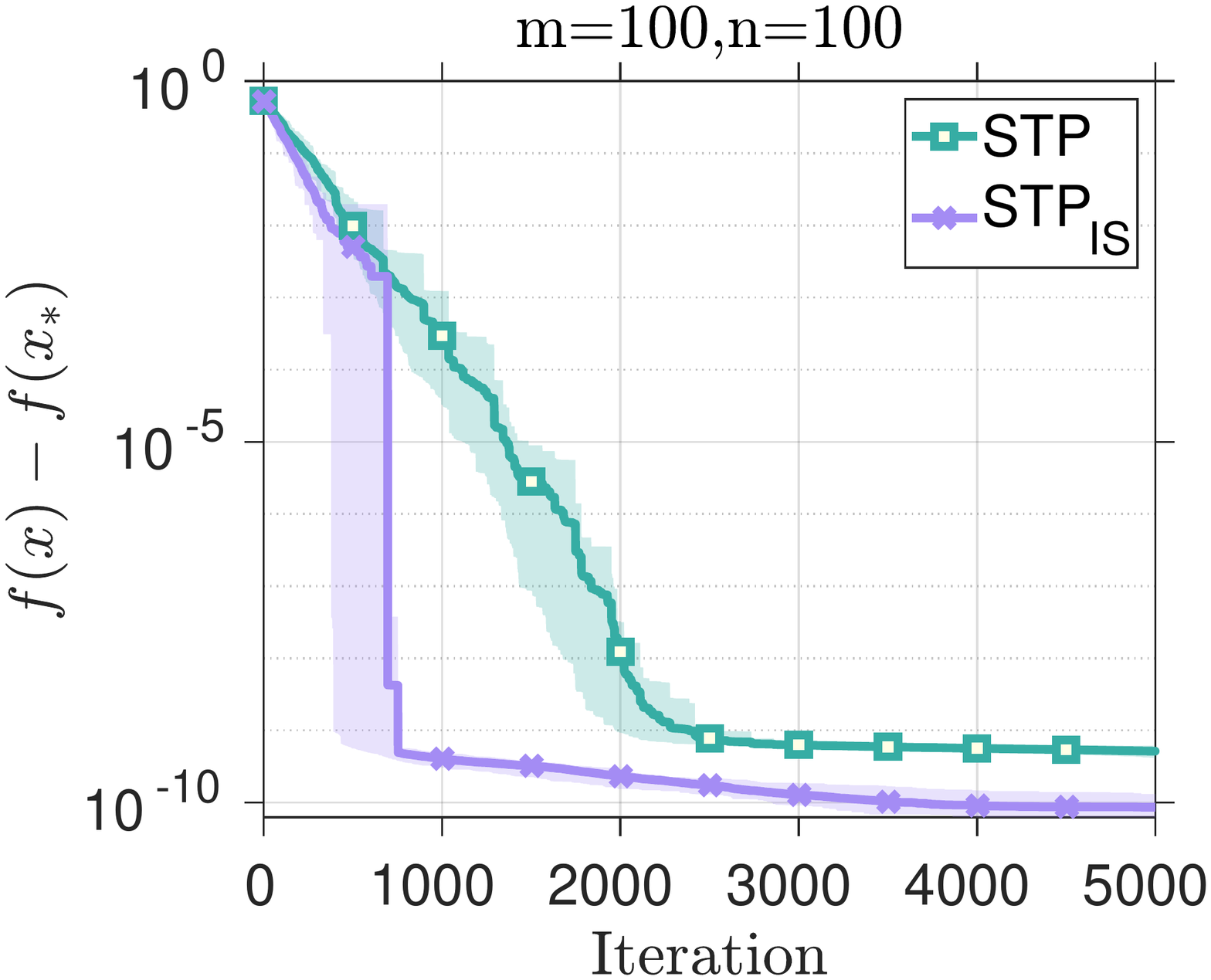} &
\includegraphics[width=0.33\textwidth]{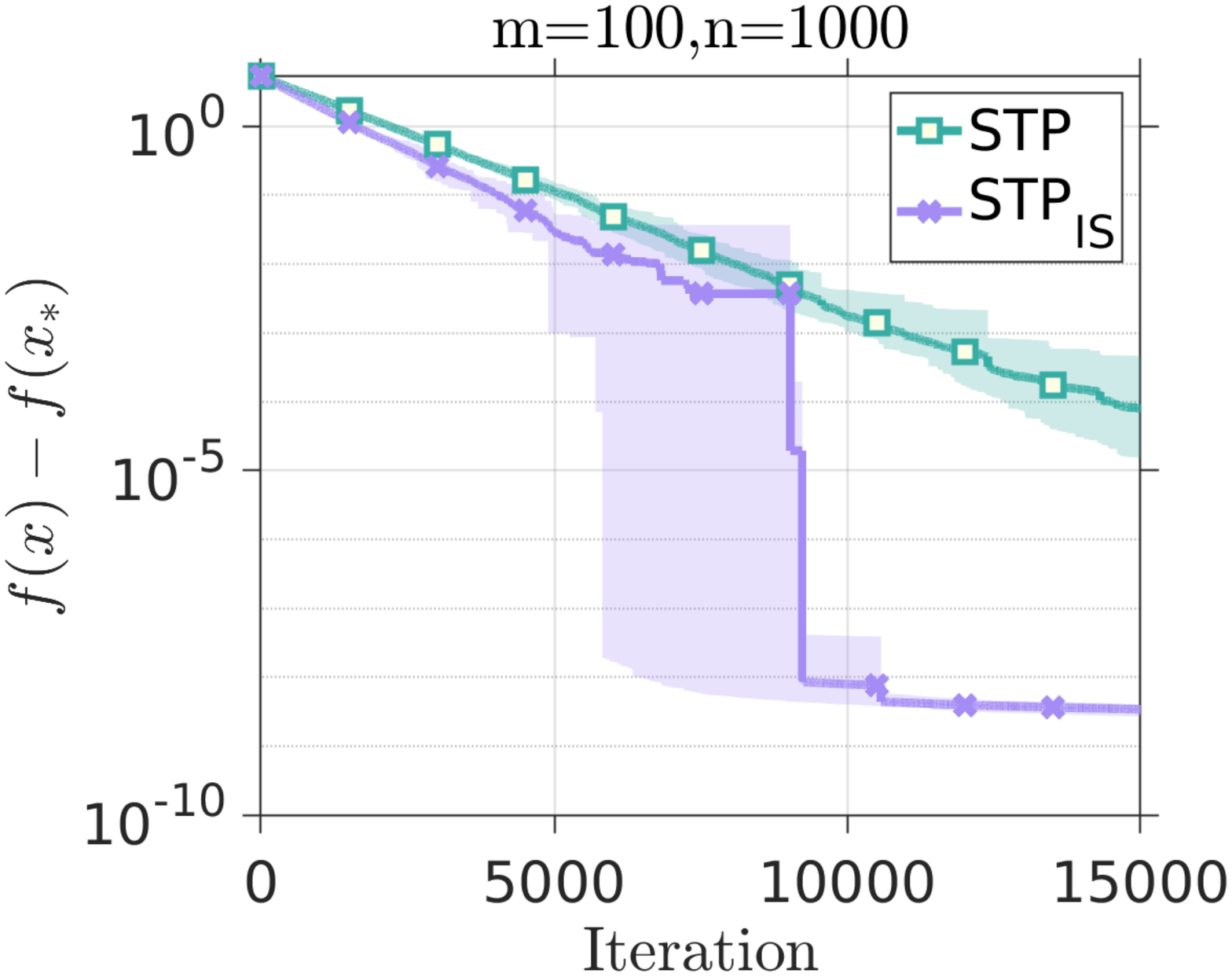}
\end{tabular}
\vspace{-1.5cm}
\caption{Shows the superiority of \texttt{STP}$_{\text{\texttt{IS}}}$ over \texttt{STP} on synthetically generated data $A$ and $y$ on the ridge regression problem. The first row shows the comparison with a varying number of rows in $A$, i.e. $m$, while the second row shows the comparison with a varying dimension $n$.}
\label{fig::ridge_regression_synthetic}
\end{figure*}

\section{Convex Case} \label{sec:conv}
This section describes our complexity results for Algorithm~\ref{alg:STP_IS} when the objective function $f$ is convex. We show that this method guarantees complexity bounds with the same order in $\epsilon$ as classical bounds in the literature, i.e., $1/\epsilon$ with an improved dependence on the Lipschitz constant. We will need the following additional assumption in the sequel.

\begin{assumption}\label{ass:level_sets} 
    The function $f$ is convex and has a bounded level set at $x_0$. That is, f satisfies:
    \begin{align*}
        R_0  \eqdef \max_x \{\|x - x_*\|_{\infty} : f(x) \leq f(x_0)\} < \infty
    \end{align*}
    Note that if $f$ is convex and has bounded level sets, the following holds:
    \begin{equation}
    \label{eq:R_0_condition}
    \begin{aligned}
        f(x) - f(x_*) & \leq \langle \nabla f(x), x - x_* \rangle \\
        &\leq \| \nabla f(x) \|_1 \|x - x_*\|_\infty \leq R_0 \| \nabla f(x) \|_1.
    \end{aligned}
    \end{equation}
\end{assumption}

\begin{theorem}\label{thm:convex2}
Let Assumptions~\ref{ass:M-smooth} and \ref{ass:level_sets} be satisfied. Choose $\alpha_{i_k} = \frac{|f(x_k + te_{i_k}) - f(x_k)|}{tv_{i_k}}$ and sufficiently small $t$ (see the \textit{appendix} for the bound on $t$). 
If  
\begin{equation}
\begin{aligned}
    k \ge  \frac{8R_0^2 n}{\min_i \frac{p_i}{v_i}} \left(\frac{1}{\epsilon} - \frac{1}{r_0} \right), 
    \label{convex_optimal_step_Size}
\end{aligned}
\end{equation}
then $\mathbb{E}\left [f(x_k) - f(x_*)\right] \leq \epsilon$.
\end{theorem}

Here, the complexity bound in \eqref{convex_optimal_step_Size} depends on the choice of the probabilities $\left\{p_i\right\}_{i=1}^{n}$ and the quantities $\left\{v_i\right\}_{i=1}^{n}$. Note that if $v_i = L_i$, it is easy to show that the minimum of the complexity bound \eqref{convex_optimal_step_Size} in $p_i$, i.e. minimizing $1/\min_i \frac{p_i}{v_i}$ in $p_i$, over a probability simplex is attained at $p_i = \frac{L_i}{\sum_{i=1}^n L_i}$. Thus, the complexity bound of this this importance sampling is
\begin{align*}
    k \ge 8 R_0^2 n \left(\sum_{i=1}^n L_i\right) \left(\frac{1}{\epsilon} - \frac{1}{r_0}\right).
\end{align*}
Since uniform sampling is proportional to $n^2L$ and that $n \sum_{i=1}^n L_i \le n^2 L$,  importance sampling is clearly better than  uniform sampling.

\section{Strongly Convex Case} \label{sec:sconv}
This section describes our complexity results for Algorithm~\ref{alg:STP_IS} when objective function $f$ is $\lambda$-strongly convex. We show that this method guarantees complexity bounds with the same order in $\epsilon$ as classical bounds in the literature, i.e., $\log(1/\epsilon)$ with an improved dependence on the Lipschitz constant. First, we define $\lambda$-strongly convexity functions.
\begin{assumption}\label{ass:f-SC} 
The function $f$ is  $\lambda$-strongly convex. That is, for some $\lambda>0$, the following holds
\begin{align*}
    f(x) \ge f(y) + \langle \nabla f(y), x-y \rangle + \frac{\lambda}{2} \|x-y\|_2^2.
\end{align*}
\end{assumption}

\begin{theorem} \label{thm:stronglyconvex2}
Let Assumptions \ref{ass:M-smooth} and \ref{ass:f-SC} be satisfied. Choose $\alpha_{i_k} = \frac{|f(x_k + te_{i_k}) - f(x_k)|}{t v_{i_k}}$ 
 and a sufficiently small $t$ (see the \textit{appendix} for the bound on $t$). If
\begin{equation}
\begin{aligned}
k \ge \frac{\max_i \frac{v_i}{p_i}}{\lambda} \log \left(\frac{2\left(f(x_0) - f(x_*)\right)}{\epsilon} \right),
\label{strogn_cvx_rate_optimal}
\end{aligned}
\end{equation}
then $\mathbb{E}\left[f(x_k) - f(x_*)\right] \leq \epsilon$.
\end{theorem}

The complexity bound in \eqref{strogn_cvx_rate_optimal} is minimized, in $p_i$ for $v_i = L_i$, with $p_i = \frac{L_i}{\sum_{i=1}^n L_i}$. Importance sampling improves over uniform sampling, since $\sum_{i=1}^n L_i \leq n L$.

\section{Experiments} \label{sec:exper}

We conduct extensive experiments on synthetic and real datasets comparing the uniform sampling \texttt{STP} against the importance sampling version \texttt{STP}$_{\text{\texttt{IS}}}$. The experiments are conducted on several choices of the function $f$. In particular, we perform experiments on regularized ridge regression on synthetic data and squared SVM loss on real data from the LIBSVM dataset \cite{chang2011libsvm}. Moreover, for non-convex problems, we compare \texttt{STP} and \texttt{STP}$_{\text{\texttt{IS}}}$ on various continuous control environments on MuJoCo~\cite{Todorov_2012}. We also compare against state-of-art solvers for the continuous control task.

\begin{figure*}[t]
\centering
\begin{tabular}{@{}c@{}c@{}c@{}}
\includegraphics[width=0.33\textwidth]{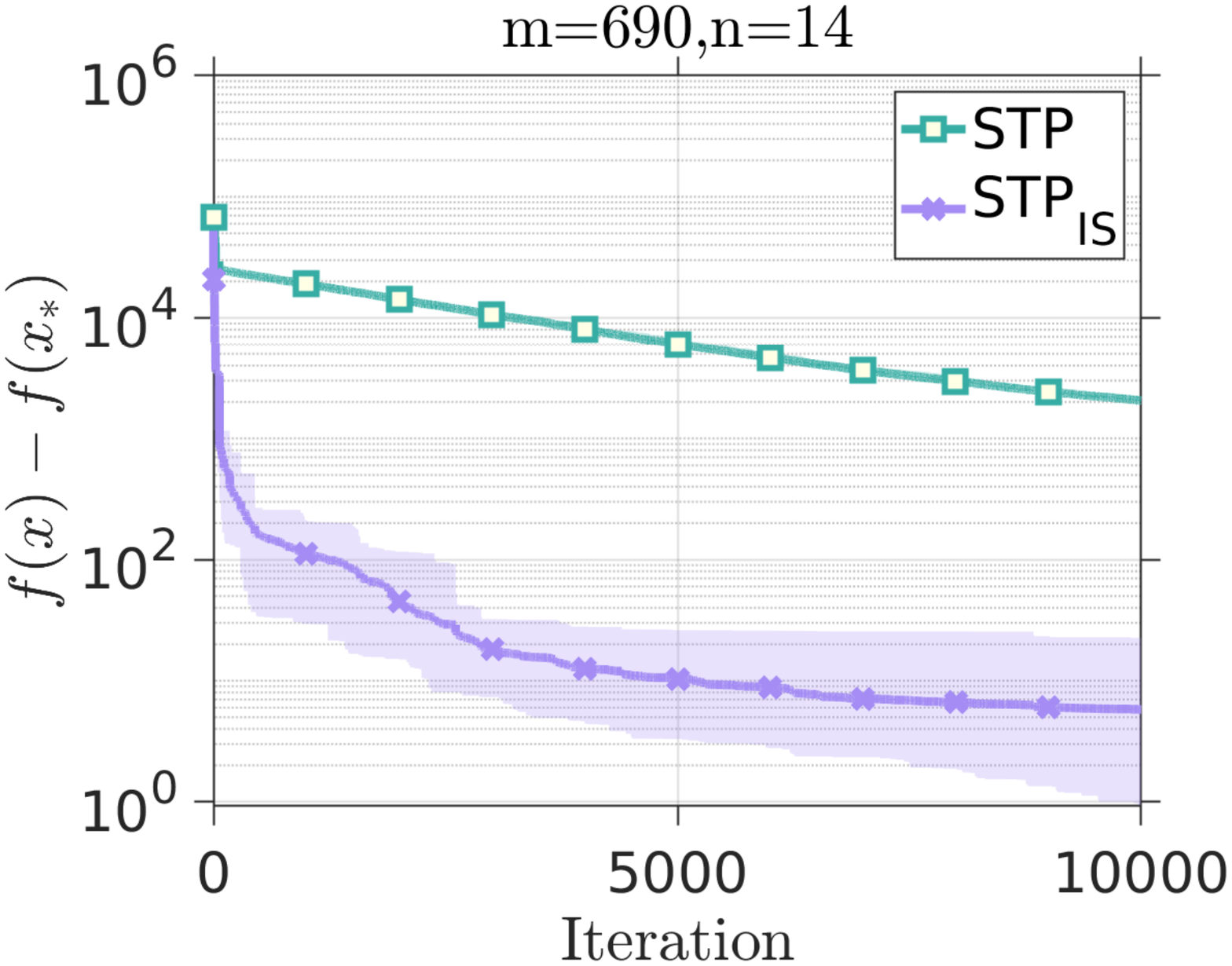}&
\includegraphics[width=0.33\textwidth]{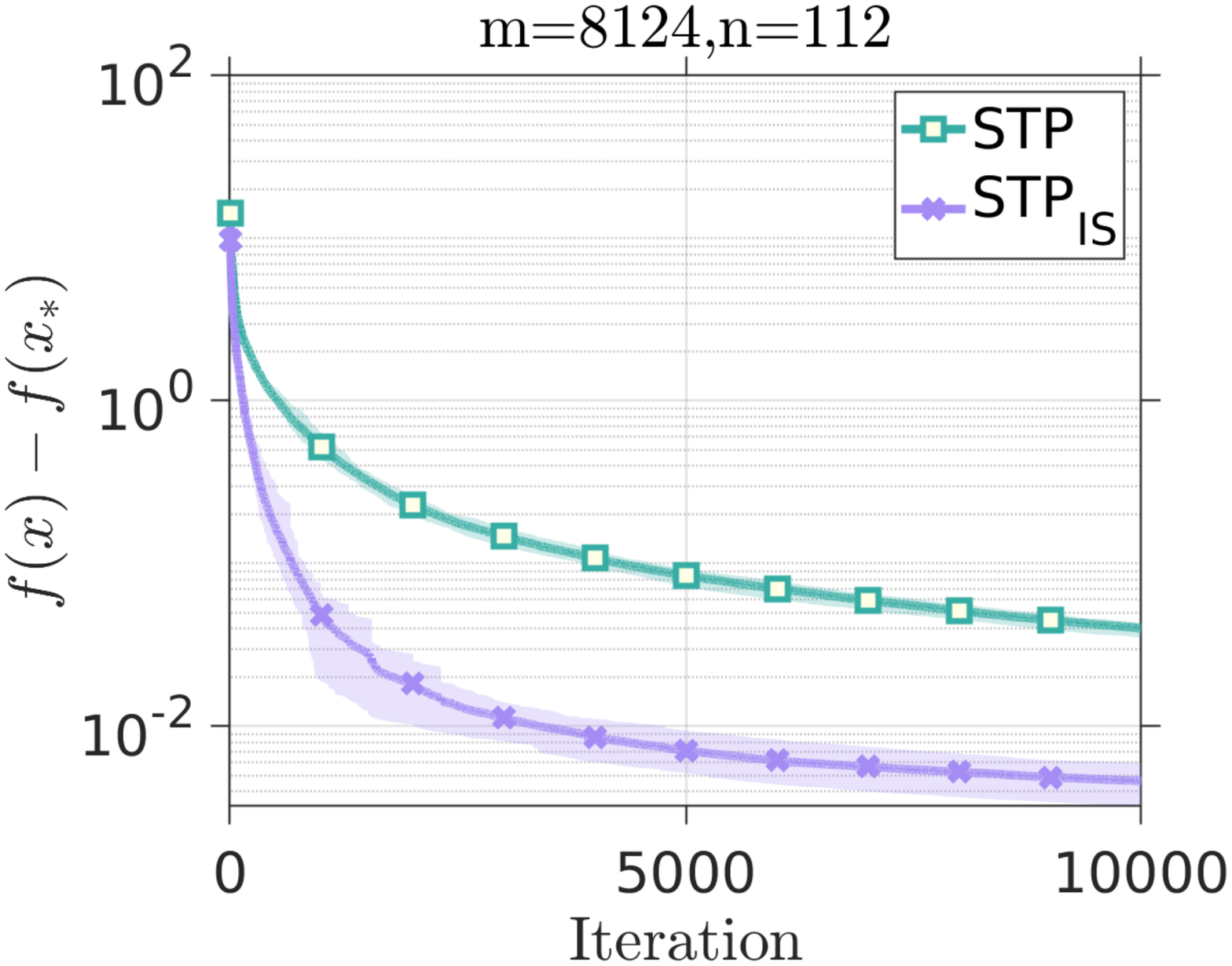}&
\includegraphics[width=0.33\textwidth]{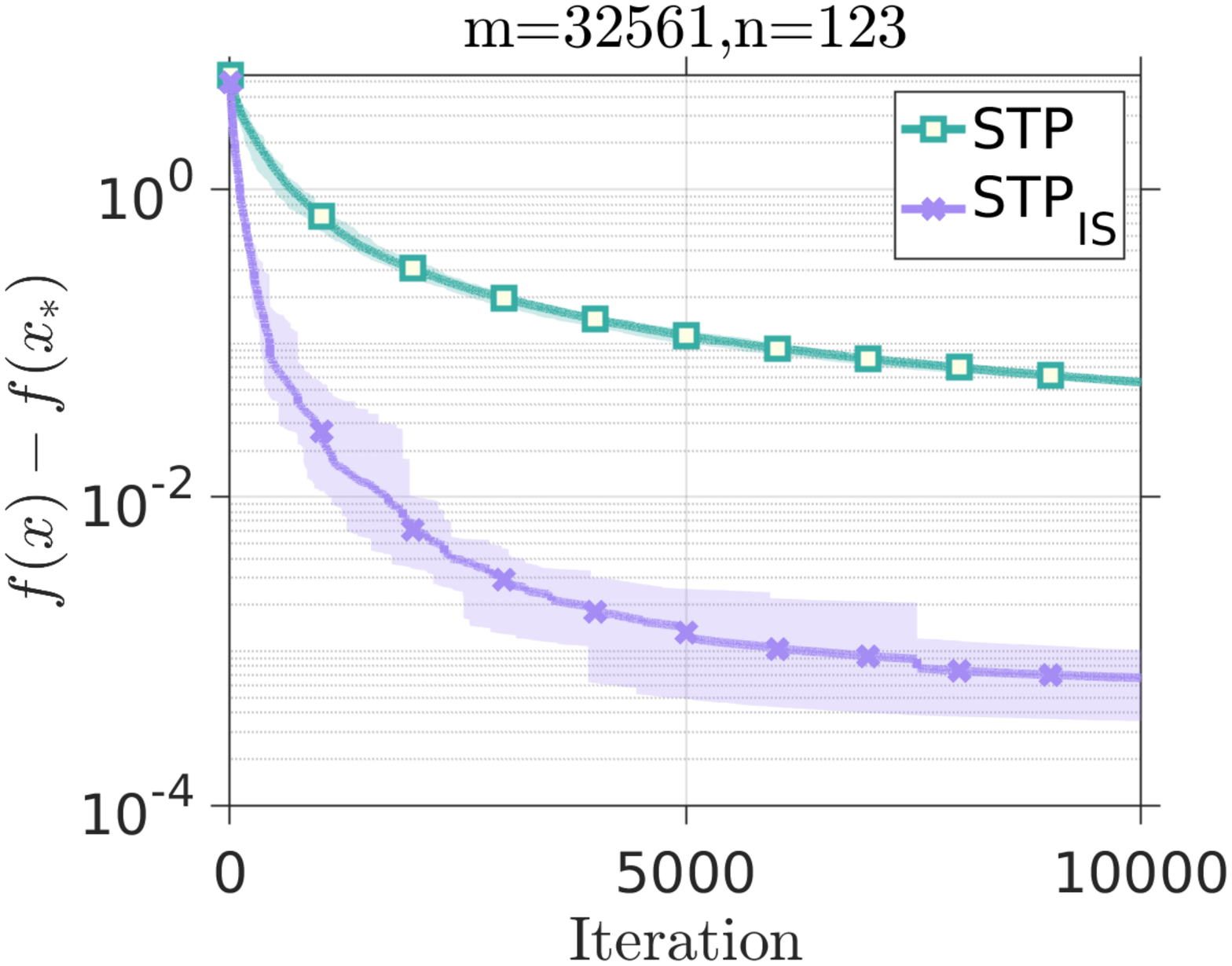}
\\ \vspace{-35mm} \\
\includegraphics[width=0.33\textwidth]{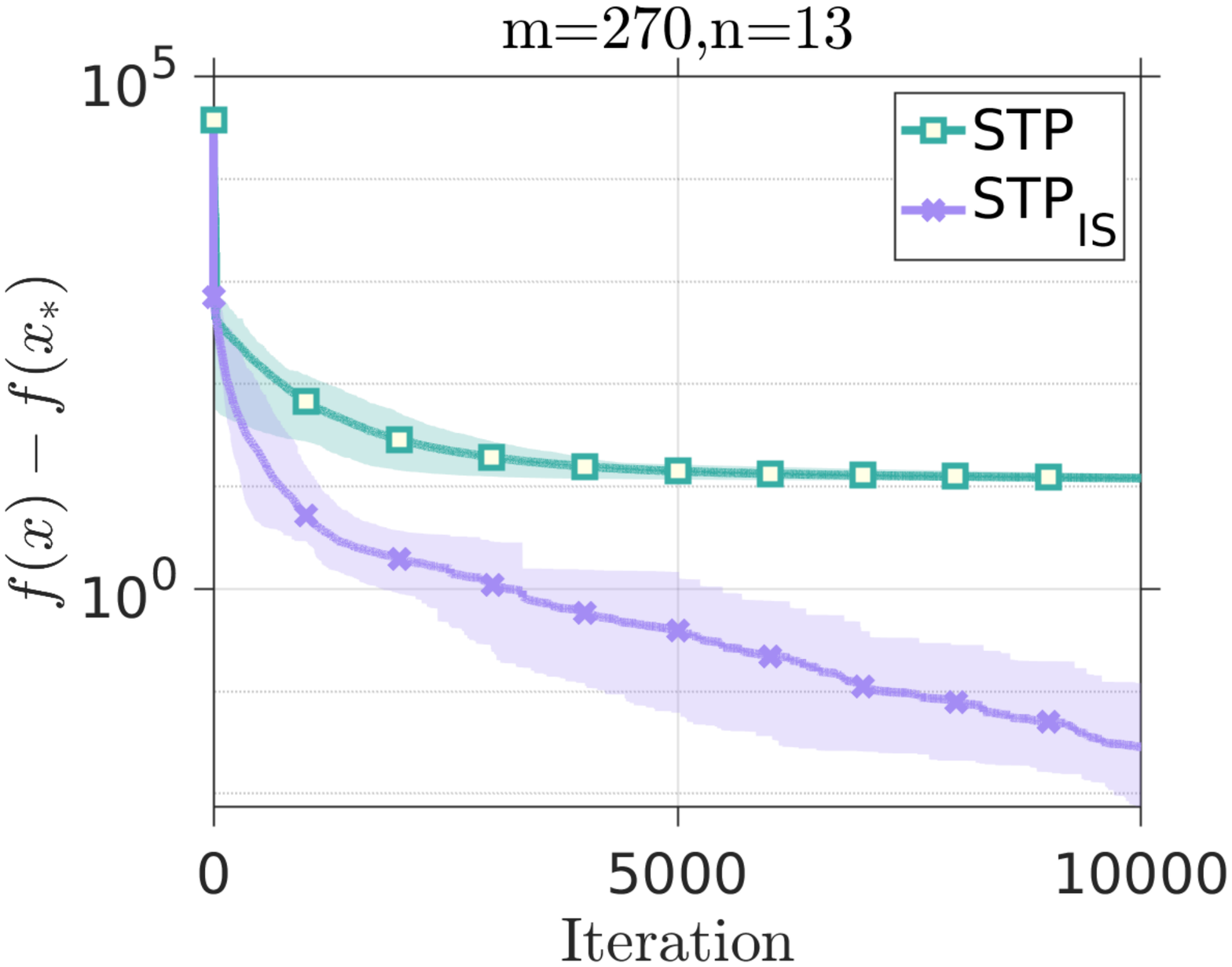} &
\includegraphics[width=0.33\textwidth]{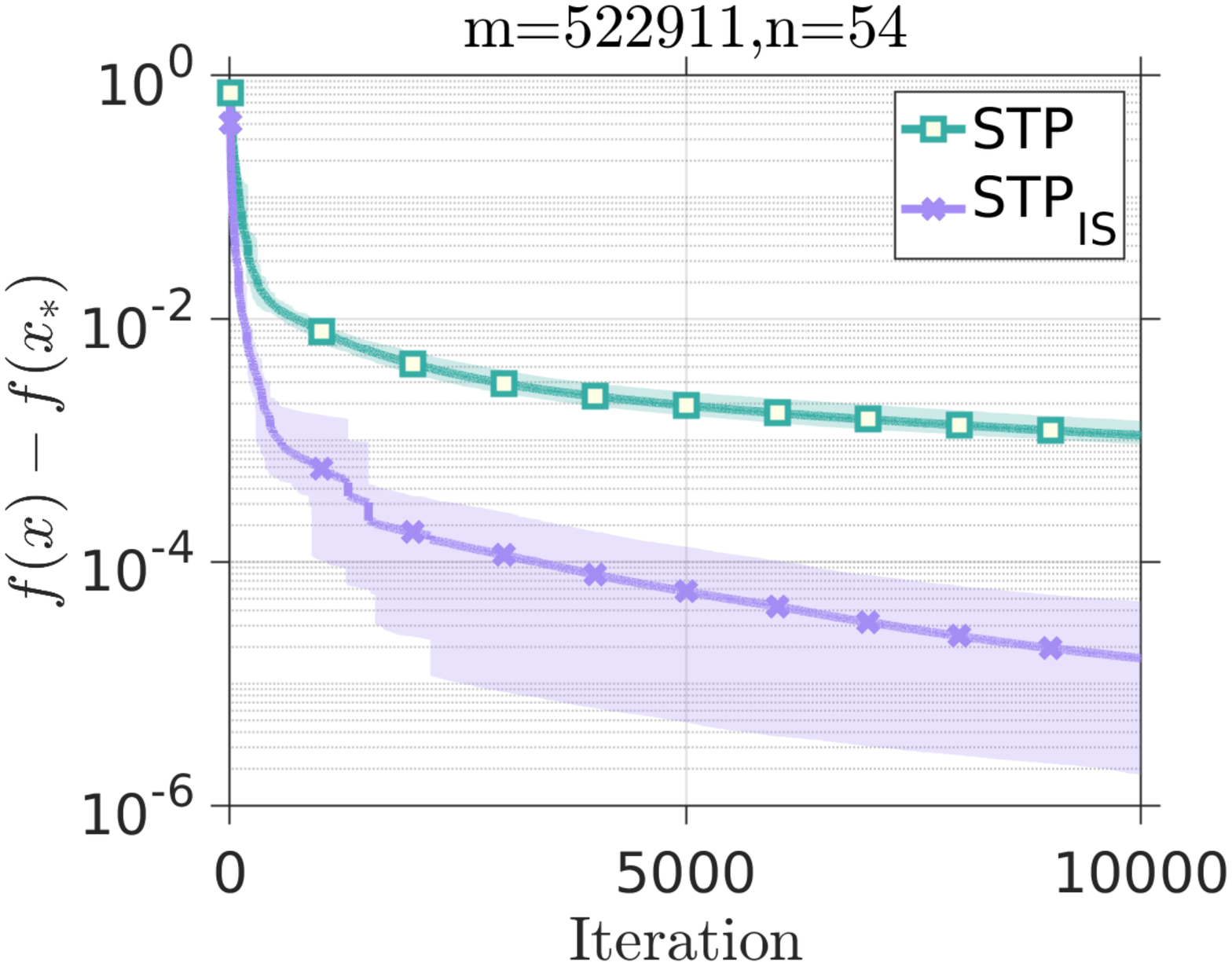} &
\includegraphics[width=0.33\textwidth]{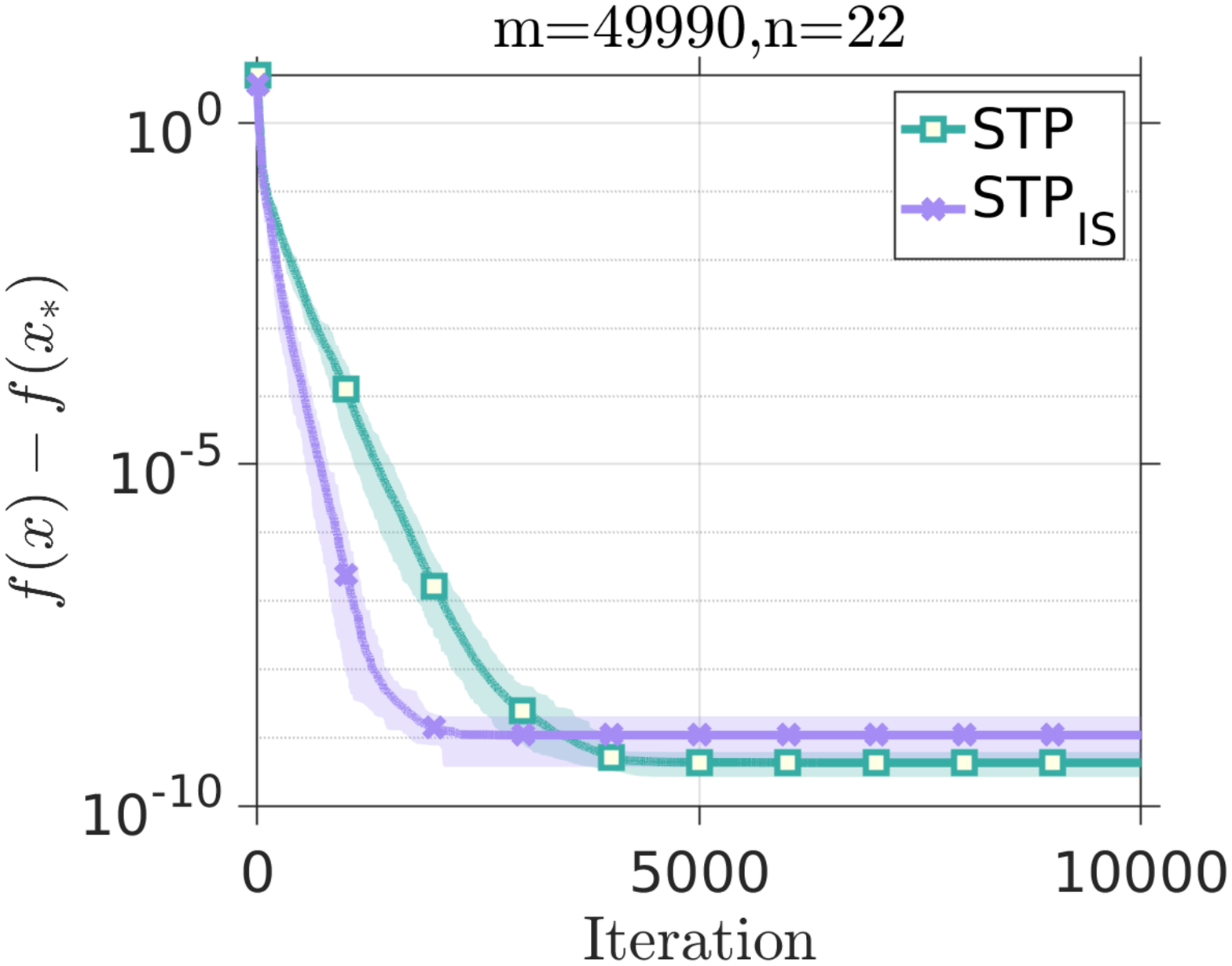}
\end{tabular}
\vspace{-1.5cm}
\caption{Shows the superiority of \texttt{STP}$_{\text{\texttt{IS}}}$ over \texttt{STP} on real LIBSVM dataset on the ridge regression problem. The datasets used in the experiments are \texttt{australian}, \texttt{mushrooms} and \texttt{a9a} in first row and \texttt{heart}, \texttt{cov1} and \texttt{ijcnn1} in second row.}
\label{fig::ridge_regression_real}
\end{figure*}

\subsection{Ridge regression on synthetic data}

We compare \texttt{STP} against \texttt{STP}$_{\text{\texttt{IS}}}$ on synthetic data on the regularized ridge regression problem:
$ f(x) = \frac{1}{2m} \|Ax - y\|_2^2 + \frac{\lambda}{2}\|x\|_2^2,$
where $A \in \mathbb{R}^{m \times n}$, $y \in \mathbb{R}^m$ are the data and $\lambda > 0$ is the regularization parameter. The elements of $A$ and $y$ were sampled from the standard Gaussian distribution $\mathcal{N}\left(0,1\right)$. Note that for ridge regression, $L_i = \frac{1}{m}\|A(:,i)\|_2^2 + \lambda$ where, following \cite{gower2018stochastic}, we normalize data such that $\|A(:,1)\|_2 = 1$ and $\|A(:,i)\|_2 = \frac{1}{m}$, $i=2,\ldots,m$ and set $\lambda = \frac{1}{m}$. We compute a high accuracy solution $x_*$ by solving the ridge regression problem exactly with a linear solver. Thereafter, the metric used is the difference between the current objective value and the optimal one, i.e. $f(x) - f(x_*)$. Since the objective is $\lambda$-strongly convex, we use the stepsize suggested by Theorem \ref{thm:stronglyconvex2}. In all experiments, we set $t = 10^{-4}$. We perform experiments across difference choices of $m$ and $n$. In the first row of Figure \ref{fig::ridge_regression_synthetic}, we compare both methods with a fixed $n=10$ and a varying $m$, i.e. $m \in \{10^{3},10^{4},10^{5}\}$. The superior performance of $\texttt{STP}_{\texttt{IS}}$ over \texttt{STP} is evident from Figure \ref{fig::ridge_regression_synthetic}. Moreover, we conduct further experiments where $m$ is fixed such that $m=100$ but with a varying dimension, i.e. $n \in \{10^{1},10^{2},10^{3}\}$. All experiments are conducted 10 times and we report the average, worst and best performances. A similar behaviour is also present as seen in the second row of Figure \ref{fig::ridge_regression_synthetic} where $\texttt{STP}_{\texttt{IS}}$ is far more superior to \texttt{STP}. In all experiments, the stopping criterion is set such that both \texttt{STP} and $\texttt{STP}_{\texttt{IS}}$ run for exactly $5 \times 10^{2}$ iterations for small problems, i.e. $n=10$, while for  problems of size $n=10^2$ and $n=10^3$, both methods are terminated at $5 \times 10^3$ and $15 \times 10^{3}$ iterations, respectively.

\subsection{Ridge regression and squared SVM on real data}

We also conduct experiments on the regularized ridge regression problem on real datasets where $A$ and $y$ are from LIBSVM data. We follow the same protocol as the experiments on synthetic data. We compare both algorithms on 6 different datasets, namely, \texttt{australian}, \texttt{mushrooms}, \texttt{a9a}, \texttt{heart}, \texttt{cov1} and \texttt{ijcnn1}. In addition to ridge regression, we conduct experiments on the same real datasets on the regularized squared SVM loss: $ f(x) = \frac{1}{2} \sum_{i=1}^m\max \left(0,1-y_i a_i^\top x\right)^2 + \frac{\lambda}{2} \|x\|_2^2,$ where $a_i$ is the $i^{\text{th}}$ row of $A$. Note that $L_i = \|A(:,i)\|_2^2 + \lambda$. Since the squared SVM problem does not exhibit a closed form solution, we compare both \texttt{STP} and $\texttt{STP}_{\texttt{IS}}$ in terms of the objective value $f(x)$. 
In Figure \ref{fig::ridge_regression_real}, we show the comparison between both \texttt{STP} and $\texttt{STP}_{\texttt{IS}}$ on the ridge regression on all 6 datasets. It is clear that using the proposed importance sampling is far more superior to standard uniform sampling. The improvement is also consistently present on the squared SVM problem as seen in Figure \ref{fig::squared_svm_loss_real}.

\begin{figure*}[t]
\vspace{-1.0cm}
\centering
\begin{tabular}{@{}c@{}c@{}c@{}}
\includegraphics[width=0.33\textwidth]{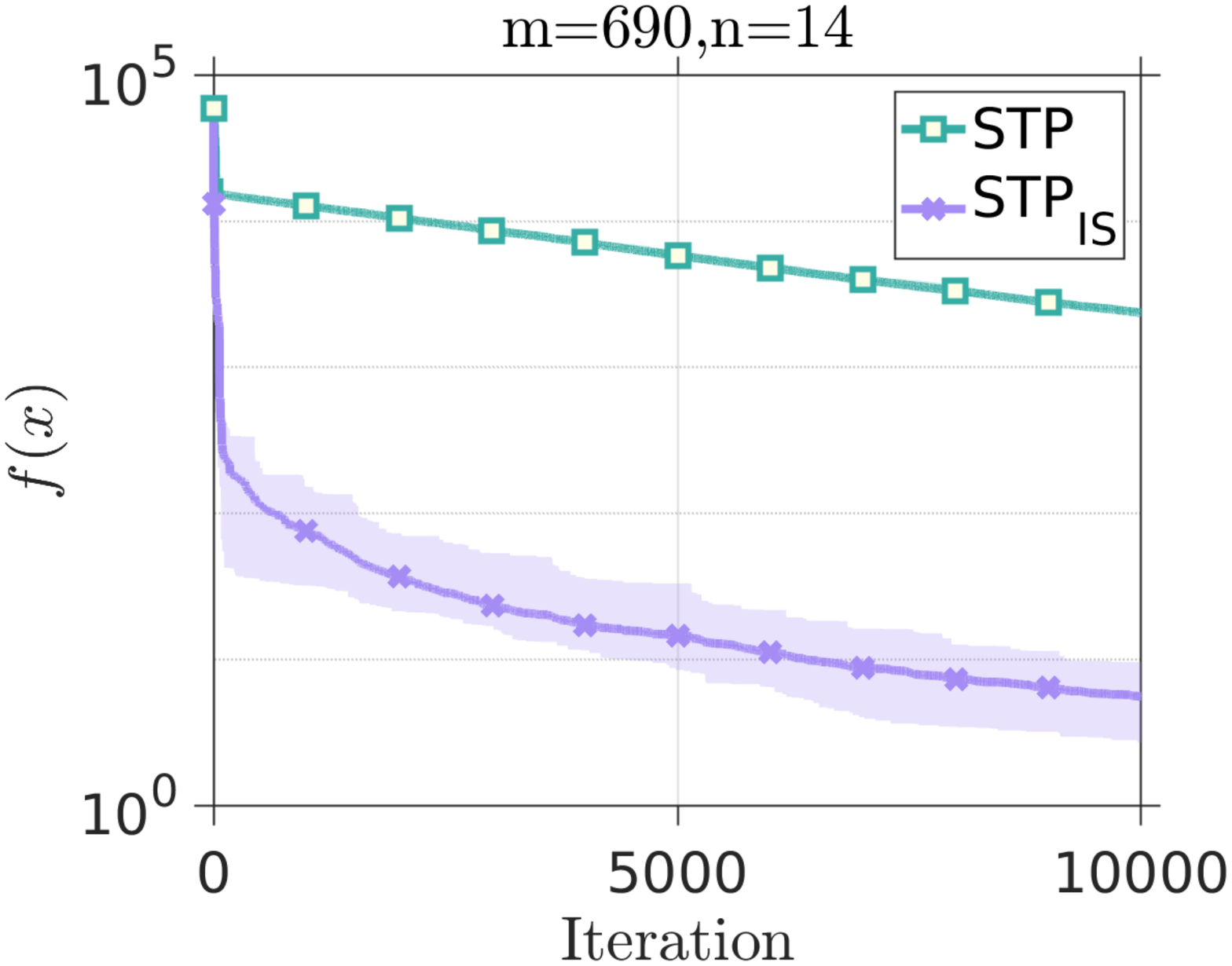}&
\includegraphics[width=0.33\textwidth]{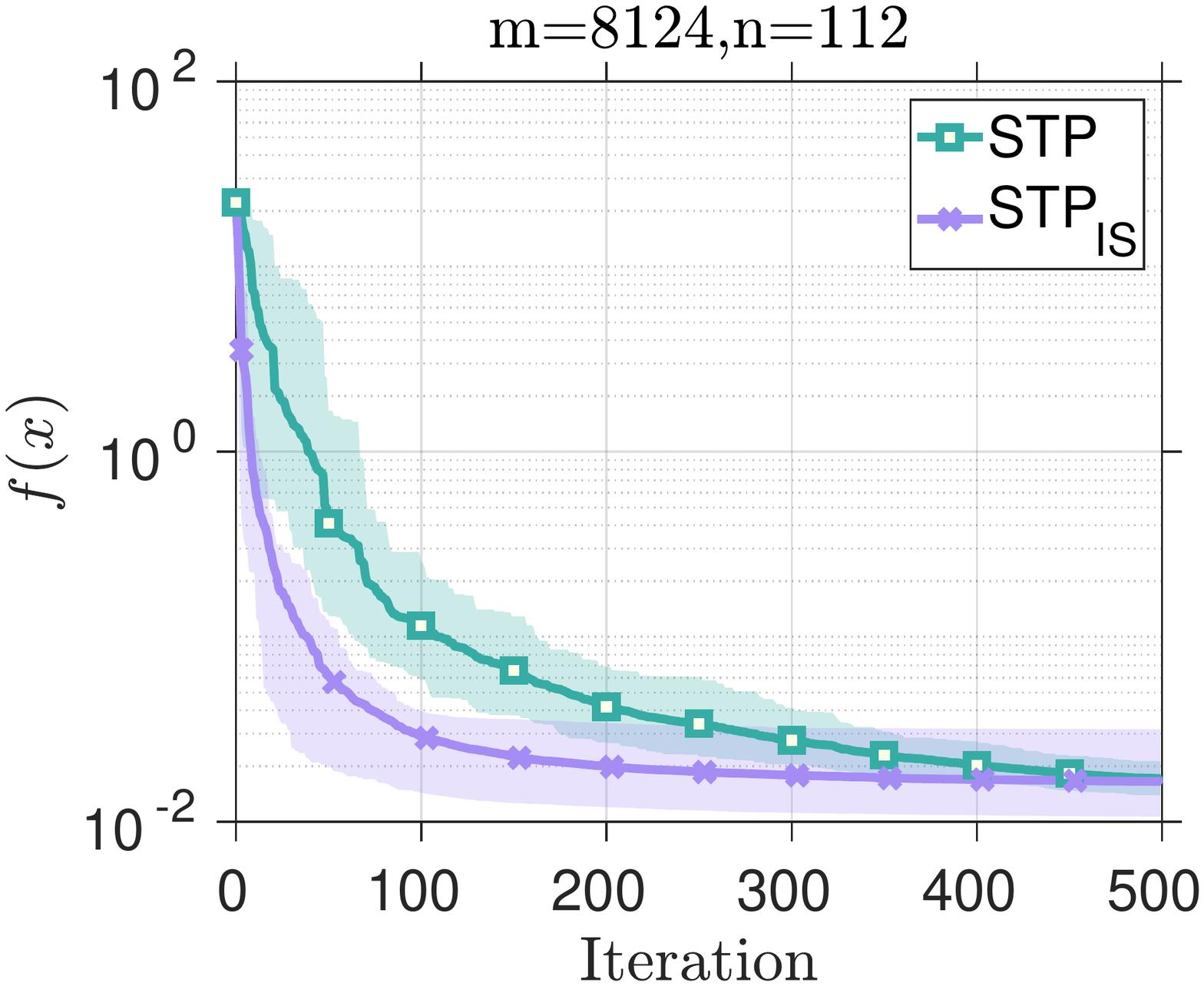}&
\includegraphics[width=0.33\textwidth]{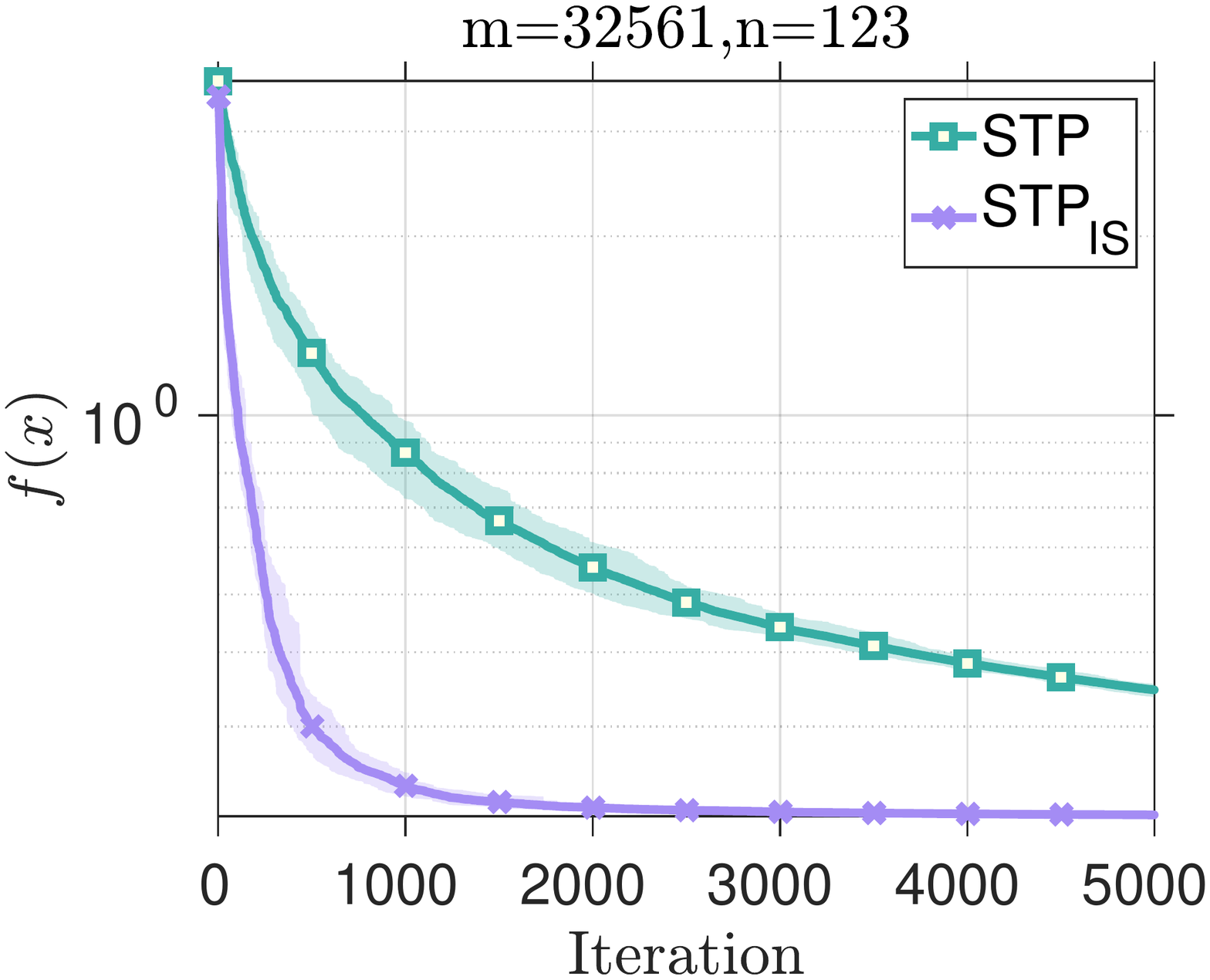} \\ \vspace{-35mm} \\
\includegraphics[width=0.33\textwidth]{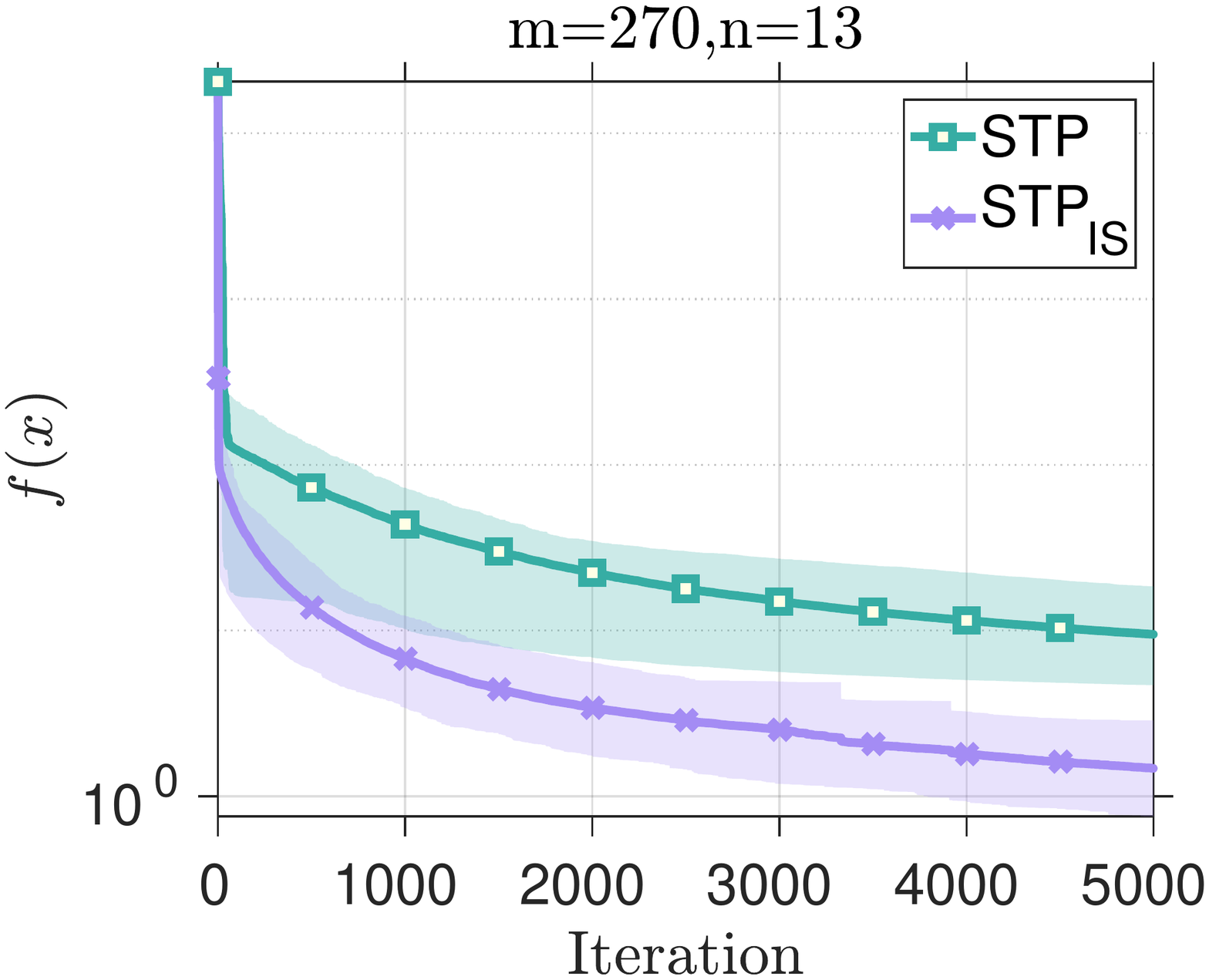} &
\includegraphics[width=0.33\textwidth]{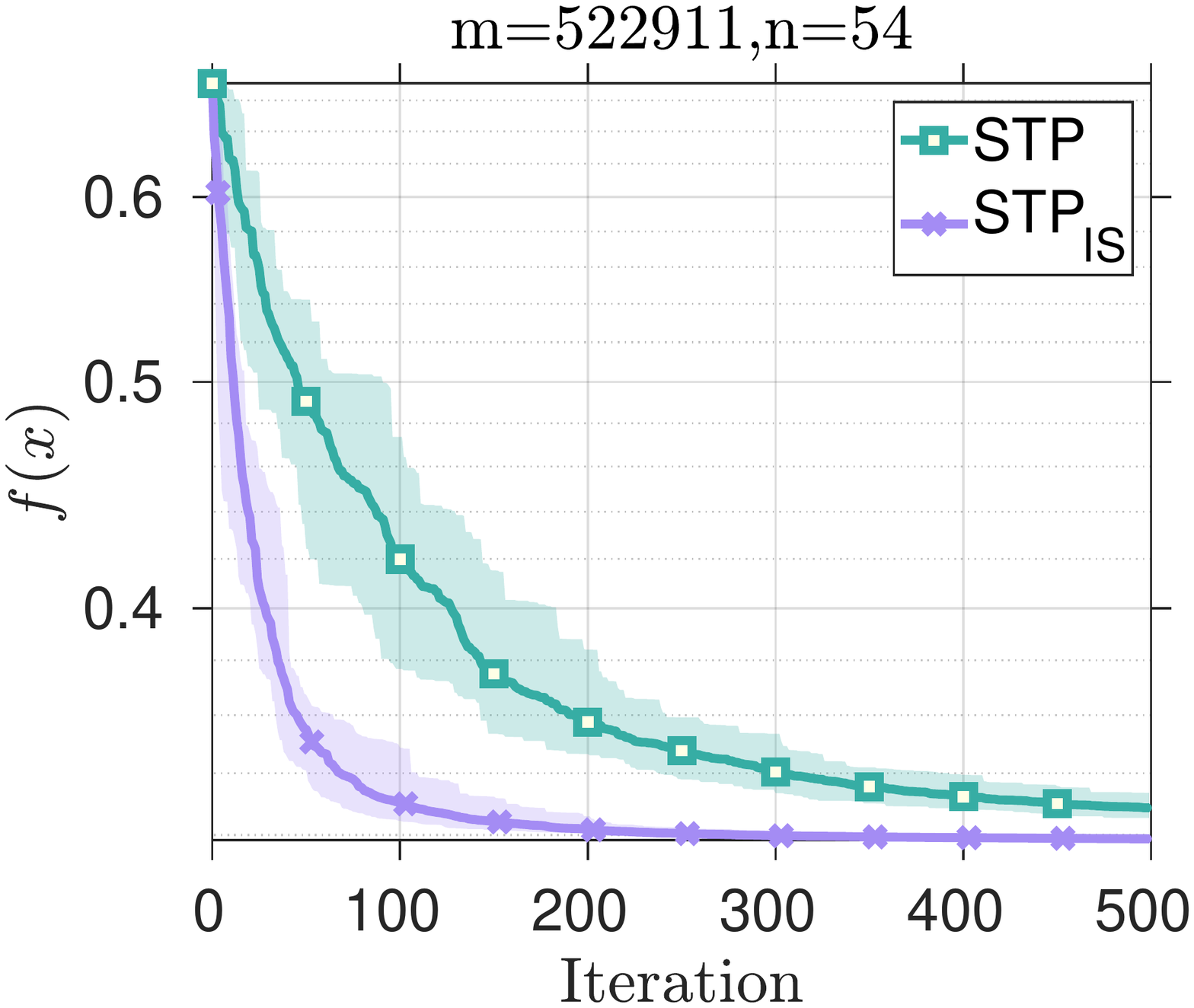} &
\includegraphics[width=0.33\textwidth]{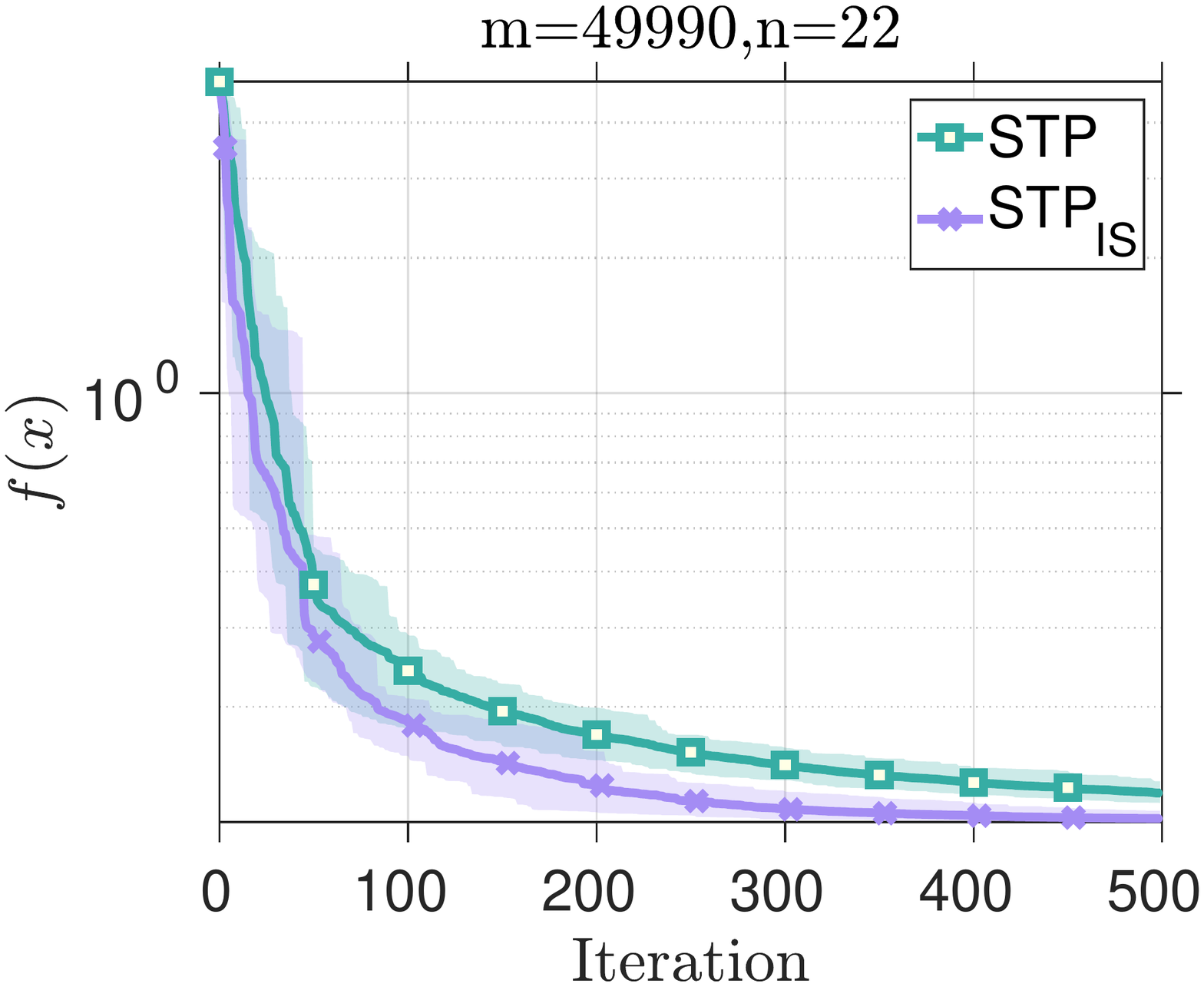}
\end{tabular}
\vspace{-1.5cm}
\caption{Shows the superiority of \texttt{STP}$_{\text{\texttt{IS}}}$ over \texttt{STP} on real LIBSVM dataset on the squared SVM loss. The datasets used in the experiments are \texttt{australian}, \texttt{mushrooms} and \texttt{a9a} in the first row and \texttt{heart}, \texttt{cov1} and \texttt{ijcnn1} for the second row.}
\label{fig::squared_svm_loss_real}
\end{figure*}

\subsection{Continuous control experiments}

Here, we address the problem of model-free control of a dynamical system. Model-free reinforcement learning algorithms (especially policy gradient methods), provide an off-the-shelf model-free approach to learn how to control a dynamical system. Such models have been typically benchmarked in a simulator. Thus, we adopt the MuJoCo~\cite{Todorov_2012} continuous control suite following its wide adaptation. We choose 5 problems with various difficulty \texttt{Swimmer-v1}, \texttt{Hopper-v1}, \texttt{HalfCheetah-v1}, \texttt{Ant-v1}, and \texttt{Humanoid-v1}. In all experiments, we use linear policies similar to \cite{Mania_2018,Rajeswaran_2017}. 

Considering the stochastic nature of the dynamical systems, i.e. $f$ is stochastic, we take multiple ($K$) measurements for $f(x_k), f(x_+)$ and $f(x_-)$ and use their mean as the function values. Considering the varying dimensionality of the state space, we use different $K$ for each problem, in particular, we set $K=2$ for \texttt{Swimmer-v1}, $K=4$ for \texttt{Hopper-v1} and \texttt{HalfCheetah-v1}, $K=40$ for \texttt{Ant-v1} and $K=120$ for \texttt{Humanoid-v1}. These values are decided using grid search over the set of $K \in \{1,2,4,8,16\}$ for low dimensional problems and $K \in \{20, 40, 120, 240\}$ for high dimensional \texttt{Ant-v1} and \texttt{Humanoid-v1} problems. Following our remark given by Equation  \ref{eq:hg89f8g9f}, we use a square matrix $B$ sampled from a standard Gaussian distribution $\mathcal{N}(0,1)$. In our experiments, this coordinate transform resulted in a better performance. Since the Lipschitz constants are not available for continuous control, we learn an estimate of the function using a parametric family (specifically multi-layer perceptron) and use its Lipschitz constants as the estimates to decide importance sampling weights. This is similar to actor-critic methods\cite{sutton2018reinforcement} used in the policy gradient literature. Similar to us, actor-critic methods learn an estimate of the value function and use it to decide which point to evaluate. We defer the details of estimation procedure to the \textit{appendix}. Following the common practice, we perform all experiments with 5 random initialization and measure the mean average reward at each iteration. We give detailed comparison of \texttt{STP} and our proposed importance sampling variant \texttt{STP}$_{\text{\texttt{IS}}}$ in terms of reward vs. sample complexity in Figure~\ref{fig::rl} for both adaptive and fixed step size cases (see Theorems \ref{thm:nonconvex1} and \ref{thm:nonconvex2}). Shaded regions in figures show standard deviations.

\begin{figure*}
\centering
\begin{tabular}{ccc}
\includegraphics[width=0.30\textwidth]{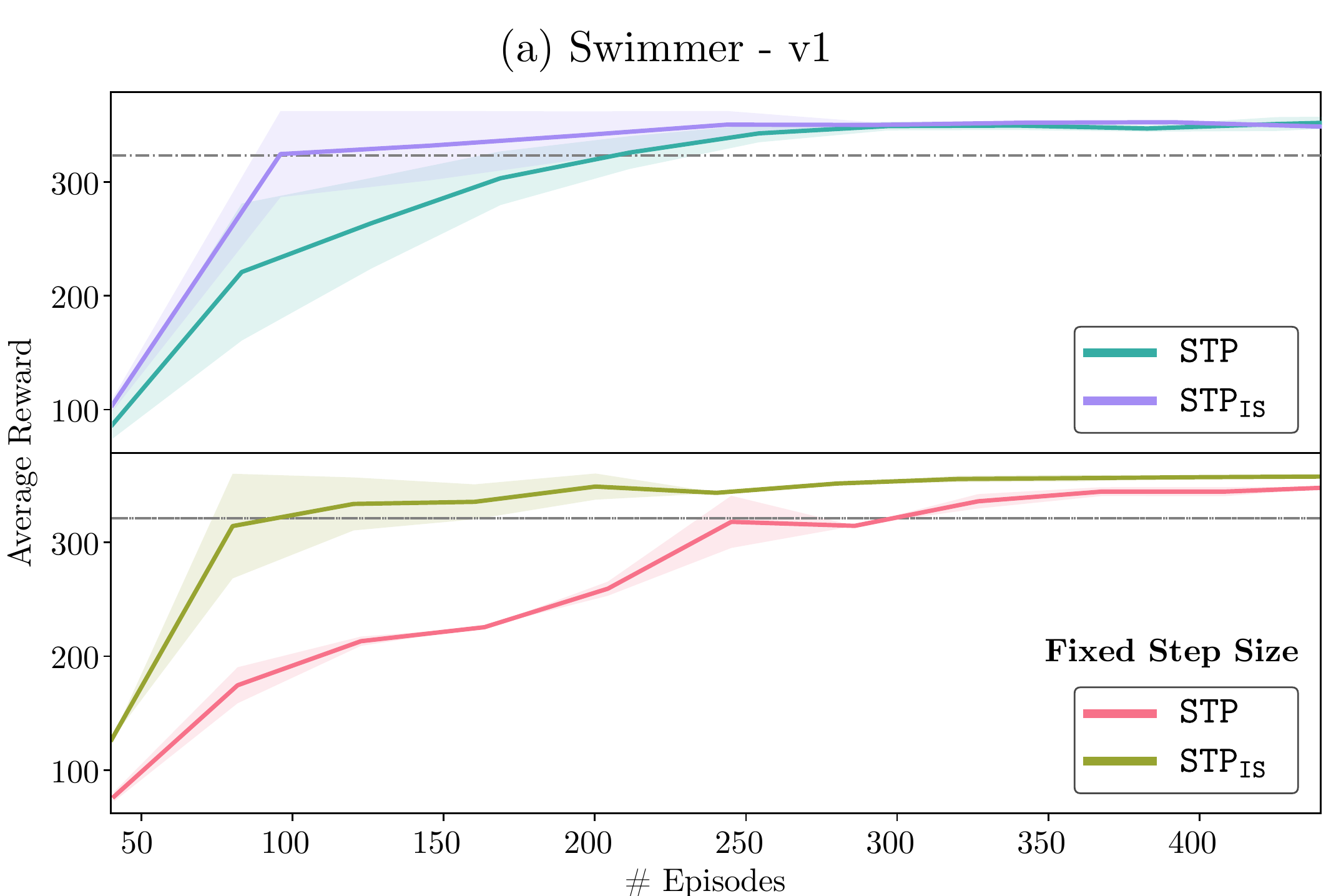} &
\includegraphics[width=0.30\textwidth]{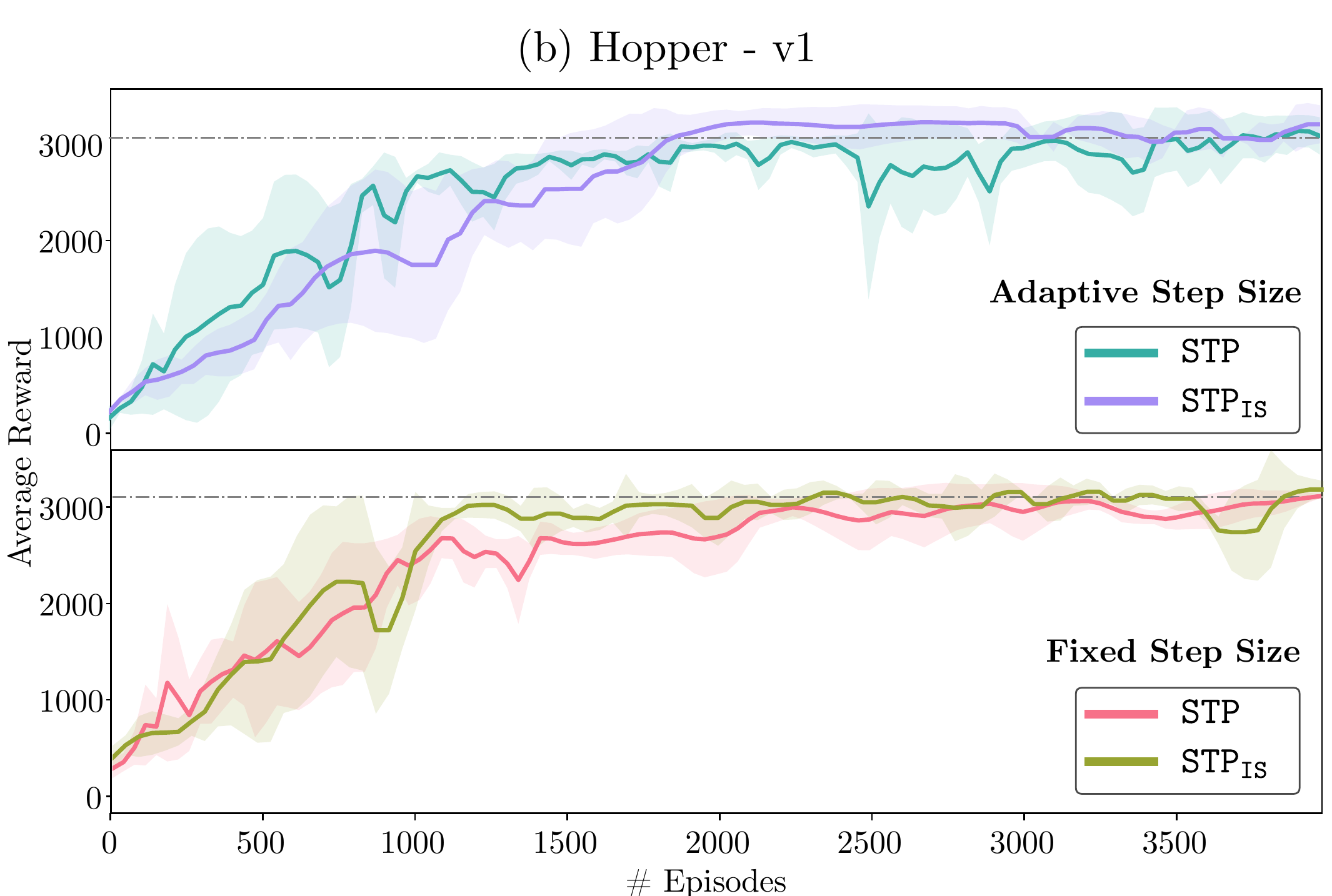} &
\includegraphics[width=0.30\textwidth]{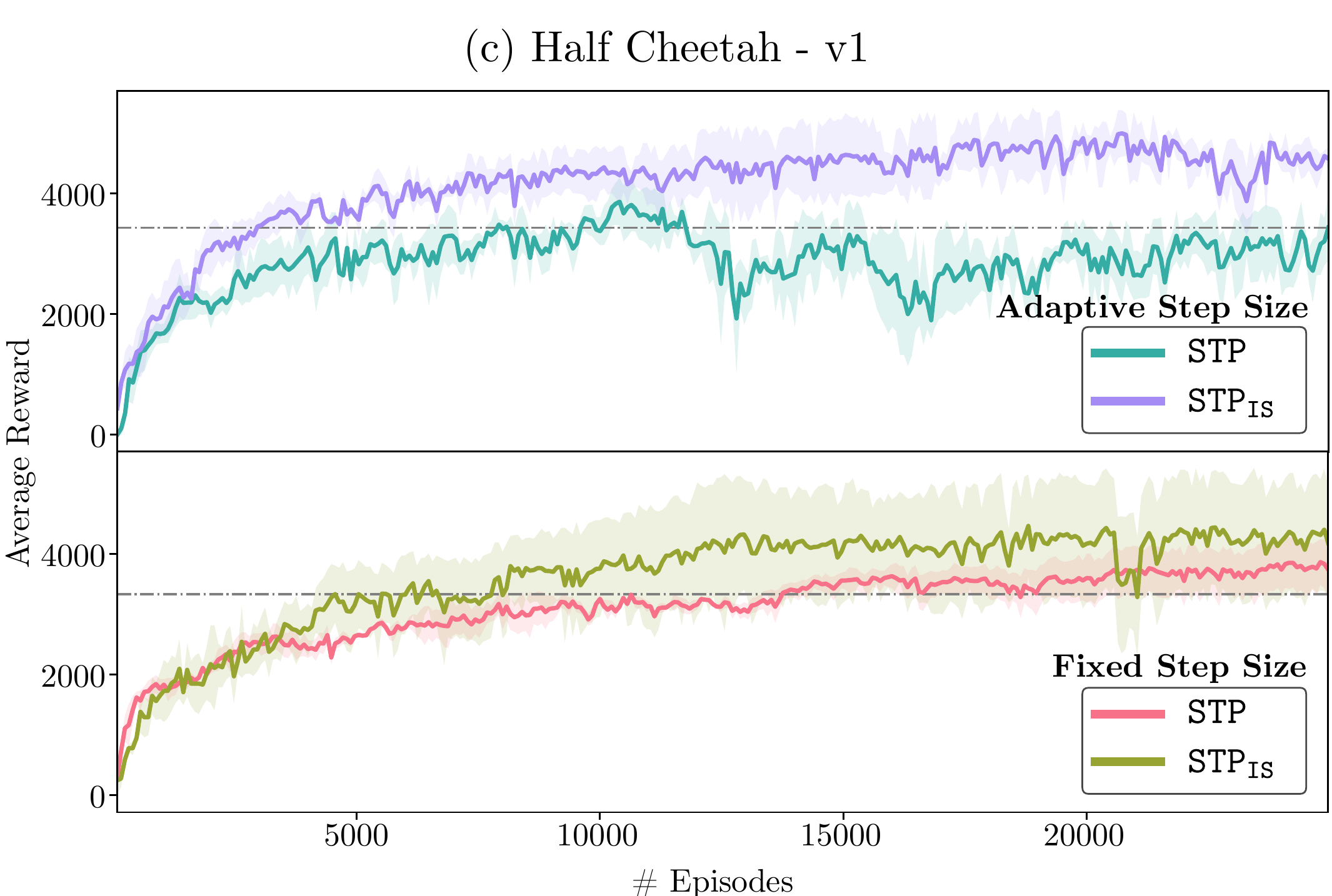}
\end{tabular}
\begin{tabular}{cc}
\includegraphics[width=0.30\textwidth]{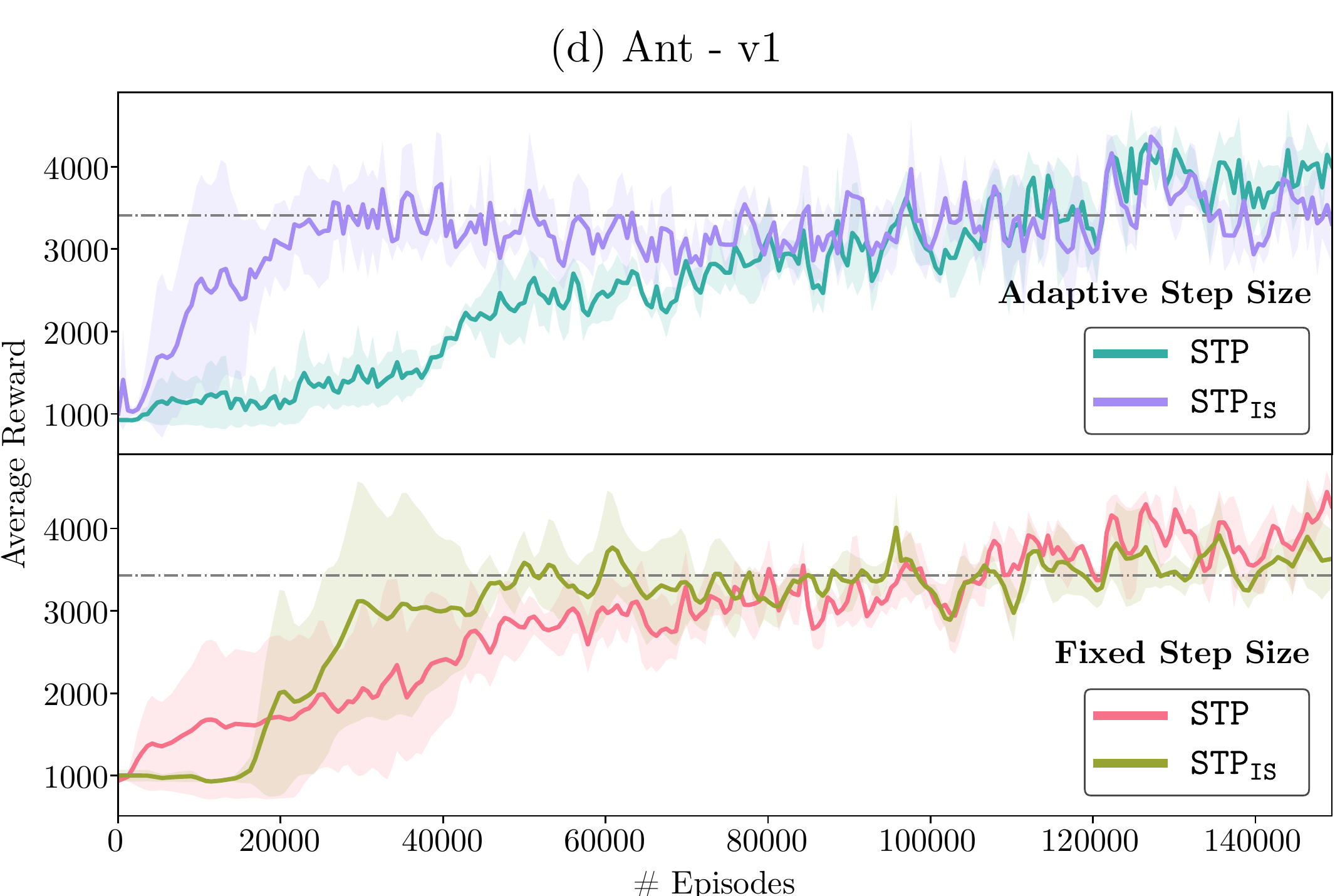} &
\includegraphics[width=0.30\textwidth]{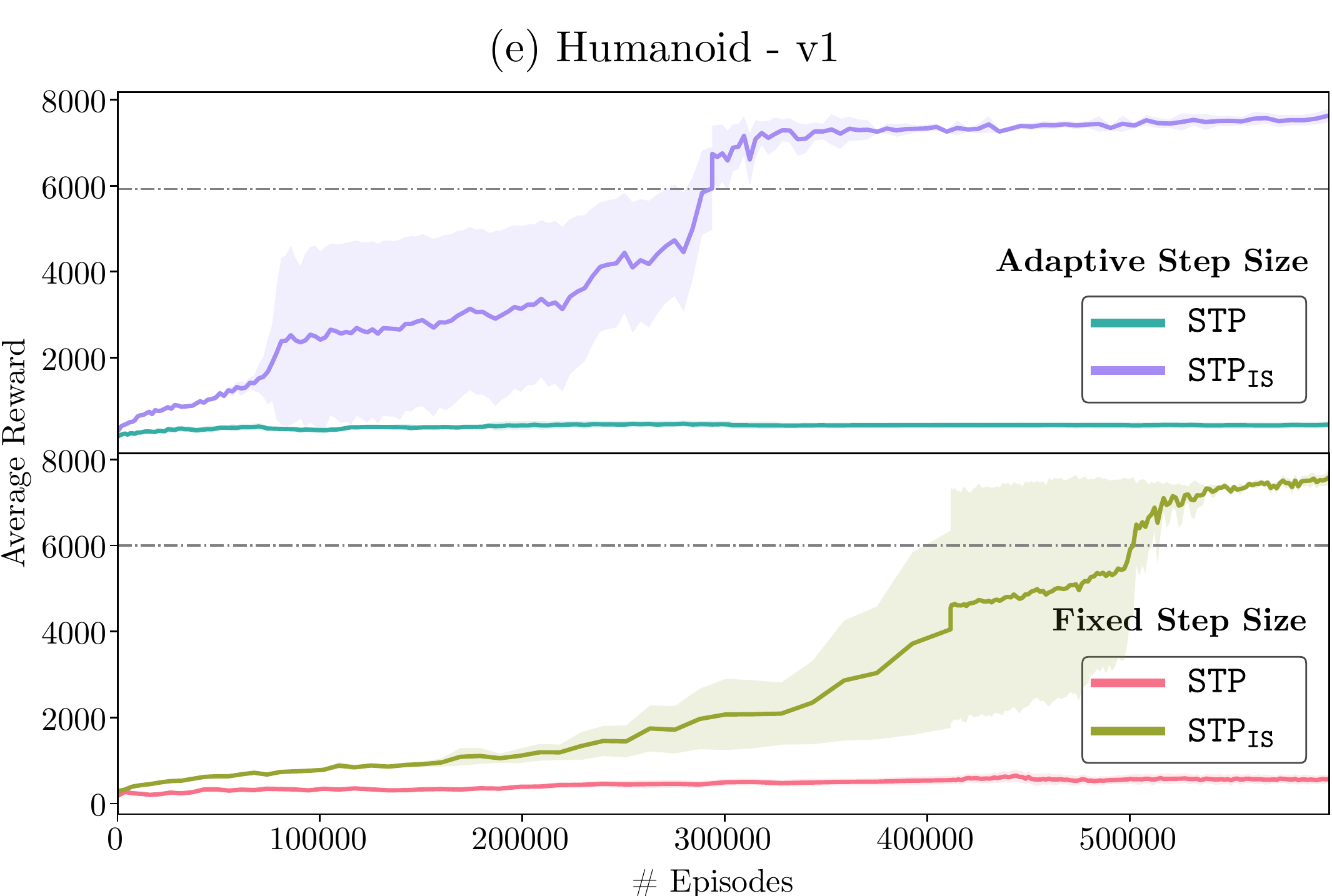} 
\end{tabular}
\caption{ \texttt{STP}$_{\text{\texttt{IS}}}$ is superior to \texttt{STP} on all 5 different MuJoCo tasks. Solid plots are the average performance over 5 runs while the shaded region shows the standard deviation. Dashed horizontal lines are the thresholds used in Table~\ref{tab:rl} to quantify sample complexity of each method.}
\label{fig::rl}
\end{figure*}

\begin{table*}[t]
\vspace{-2mm}
\caption{For each MuJoCo task, we report the average number of episodes required to achieve a predefined reward threshold. Results for our method is averaged over five random seeds, the rest is copied from \cite{Mania_2018} (N/A means the method failed to reach the threshold. UNK means the results is unknown since they are not reported in the literature.)}
  \centering
 \resizebox{\textwidth}{!}{
\begin{tabular}{rccccccccc}
  \toprule
    &  & \multicolumn{2}{c}{Fixed Step Size} & \multicolumn{2}{c}{Adaptive Step Size} & & & & \\
      \cmidrule(lr){3-4} \cmidrule(lr){5-6}
    & Threshold	& \texttt{STP} & \texttt{STP}$_{\text{\texttt{IS}}}$ & \texttt{STP} & \texttt{STP}$_{\text{\texttt{IS}}}$ & ARS(V1-t) & ARS(V2-t) &	NG-lin	& TRPO-nn \\
    \midrule
    \texttt{Swimmer-v1} & 325 & 320 & 110 & 200 &    90 & 100 & 427 & 1450 & N/A  \\
    \texttt{Hopper-v1} & 3120 & 3970 & 2400  & 3720 &  1870  & 51840  & 1973 & 13920 & 10000 \\
    \texttt{HalfCheetah-v1} & 3430 & 13760 & 4420 & 5040 & 2710 & 8106 & 1707 & 11250 & 4250  \\
    \texttt{Ant-v1} & 3580 & 107220 & 43860 &96980 &26480 & 58133 & 20800  & 39240 & 73500 \\
    \texttt{Humanoid-v1} & 6000 & N/A & 530200 & N/A & 296800 & N/A & 142600  & 130000  & UNK \\
      \bottomrule
\end{tabular}}
\label{tab:rl}
\end{table*}

As seen from Figure~\ref{fig::rl}, our proposed importance sampling version \texttt{STP}$_{\text{\texttt{IS}}}$ significantly improves sample complexity when compared to \texttt{STP}. Moreover, the difference is significant for high dimensional problems like HalfCheetah, Ant and Humanoid. The results suggest that \texttt{STP} fails to scale to very high dimensional problems like Humanoid. Our method tackles this and improves the sample complexity of \texttt{STP}. Such results also suggest that it is feasible to estimate the coordinate-wise Lipschitz gradient constants, detailed in Assumption \ref{ass:M-smooth}, of a complicated non-convex function using a data-driven approach. An interesting conclusion from Figure~\ref{fig::rl} is that adaptive step size performs better than fixed step size even after a large hyper-parameter search, particularly, for higher dimensional problems.

In order to compare our method with the existing state-of-art DFO and policy gradient methods, we also tabulate the sample complexity of our method and several existing baselines. Similar to \cite{Mania_2018}, we compute the average number of episodes needed to reach a predefined threshold. Although there are many DFO and policy gradient methods in literature, we report ARS \cite{Mania_2018} as a representative DFO method since it outperforms other baselines. As for policy gradient approaches, we report TRPO \cite{schulman2015trust} as a representative policy gradient method since it is widely used in the community. Moreover, we use NG \cite{Rajeswaran_2017} as a policy gradient method using linear policies. We list sample complexity results for all methods and tasks in Table~\ref{tab:rl}.

As the results suggest, \texttt{STP} is competitive with existing solutions for low dimensional problems (Swimmer, Hopper and HalfCheetah) whereas it underperforms existing solutions for Ant and fails to solve the Humanoid problem. Our theoretically proposed importance sampling version \texttt{STP}$_{\text{\texttt{IS}}}$ significantly improves \texttt{STP} and results in a performance either competitive with or better than existing baselines in all problems except for Humanoid. Although our method successfully solves the Humanoid problem, it has worse sample complexity than other solutions. Hence, scaling \texttt{STP} to very large dimensional continuous control problems (e.g. \texttt{Humanoid-v1} state space has more than 1000 dimensions) is still an open problem. Moreover, for lower dimensional problems like (Swimmer and Hopper), our method outperforms all existing methods.

An interesting question is whether we can use first order methods utilizing the estimated function. In order to estimate the Lipschitz smoothness constant of the function, we utilize a parametric family and one can argue that its gradients can be used as a surrogate gradient in first order optimization. We compare our method with a simple first order baseline using the surrogate gradient in SGD; it fails in all environments. Hence, we do not show the results in the paper.

\section{Conclusion} \label{sec:conc}
We propose and analyze a DFO algorithm with importance sampling \texttt{STP}$_{\text{\texttt{IS}}}$ enjoying the best known complexity bounds known in DFO literature. Experiments on ridge regression and squared SVM objectives for both synthetic and LIBSVM datasets demonstrate the superiority of \texttt{STP}$_{\text{\texttt{IS}}}$ over its uniform version. We also conduct experiments on a collection of continuous control tasks on several MuJoCO environments. We are orders of magnitudes better than the uniform sampling and comparable or better than the state-of-art methods.

\noindent\textbf{Acknowledgments.} This work was supported by the King Abdullah University of Science and Technology (KAUST) Office of Sponsored Research.

\bibliographystyle{aaai}
\bibliography{references.bib}

\appendix
\onecolumn

\section{Preliminaries}\label{sec:preliminaries}

We establish the key lemma which will be used to prove the theorems stated in the paper.
\begin{lemma}\label{lemma:key_lem} 
If $f$ satisfies Assumption \ref{ass:M-smooth} and following the \texttt{STP}$_{\text{\texttt{IS}}}$ update, the following holds:
  \begin{align*}   
    \mathbb{E}\left[f(x_{k+1}) \;|\; x_k\right] \leq f(x_k) -  \mathbb{E} \left[\alpha_{i_k} | \langle \nabla f(x_k), e_{i_k} \rangle| \;|\; x_k\right] + \frac{1}{2} \mathbb{E} \left[\alpha_{i_k}^2 L_{i_k} \;|\; x_k\right].
  \end{align*}
\end{lemma} 
\begin{proof}
Since
\begin{align*}
    f(x_{k+1}) &\leq \min \{f(x_k + \alpha_{i_k} e_{i_k}), f(x_k - \alpha_{i_k} e_{i_k})\}  \leq f(x_k) - |\alpha_{i_k} \langle \nabla f(x_k),e_{i_k} \rangle | + \frac{\alpha_{i_k}^2 L_{i_k}}{2}.
\end{align*}
Then the result follows by taking conditional expectation on $x_k$.
\end{proof}

\section{Non convex case}

\begin{thm} Let Assumption~\ref{ass:M-smooth} hold. Choose $\alpha_{i_k}=\tfrac{\alpha_0}{v_{i_k} \sqrt{k+1}}$, where $\alpha_0>0$. If
\begin{equation}\label{eq:isjss8sus-} 
K\geq \frac{2 \left( \frac{\sqrt{2}(f(x_0)-f_*)}{\alpha_0} + \frac{\alpha_0 \sum_{i=1}^n  \frac{p_i {L_i}}{{v_{i}^2}}}{2} \right)^2}{\left(\min_i \frac{p_i}{v_i}\right)^2 \epsilon^2},
\end{equation}
then $\displaystyle \min_{k=0,1,\dots,K} \mathbb{E} \left[ \|\nabla f(x_k)\|_{1} \right] \leq \epsilon.$
\end{thm}
\begin{proof}
We have from Lemma~\ref{lemma:key_lem} 
\begin{equation} \label{eq:condexp}
\mathbb{E} \left[ f(x_{k+1}) \;|\; x_k \right] \leq f(x_k)-  \underbrace{\mathbb{E} \left[ \alpha_{i_k} |\ve{ \nabla f(x_k)}{e_{i_k}}| \;|\; x_k \right]}_{\circledOne} + \underbrace{\frac{1}{2} \mathbb{E} \left[\alpha_{i_k}^2 L_{i_k} \;|\; x_k\right]}_{\circledTwo},
\end{equation}

From \circledOne we have 
\begin{eqnarray*}
\mathbb{E} \left[ \alpha_{i_k} |\ve{\nabla f(x_k)}{e_{i_k}}| \;|\; x_k \right] &= \frac{\alpha_0}{\sqrt{k+1}} \sum_{i=1}^n \frac{p_i}{{v_{i}}} |\ve{\nabla f(x_k)}{e_{i_k}}| \ge \frac{\alpha_0}{\sqrt{k+1}} \min_i\frac{p_i}{v_i} \|\nabla f(x_k)\|_1.
\end{eqnarray*}

From \circledTwo we have
\begin{eqnarray*}
\frac{1}{2} \mathbb{E} \left[ \alpha_{i_k}^2 L_{i_k} | x_k \right] &= &\frac{1}{2} \frac{\alpha_0^2}{{k+1}} \sum_{i=1}^n \frac{p_i {L_i}}{{v_{i}^2}}
\end{eqnarray*}

By injecting \circledOne and \circledTwo in (\ref{eq:condexp}) and taking the expectation we get 
\begin{equation} 
\theta_{k+1}  \leq \theta_{k} - \frac{\alpha_0 \min_i\frac{p_i}{v_i}}{ \sqrt{k+1}} g_k  + \frac{\alpha_0^2 \sum_{i=1}^n \frac{p_i L_i}{v_i^2}}{2({k+1})},
\end{equation}
where $\theta_{k} = \mathbb{E} \left[ f(x_k) \right]$  and $g_k = \mathbb{E} \left[ \|\nabla f(x_k)\|_1 \right]$. By re-arranging the terms we get:

\begin{equation}
\label{eq:s9jd7d76d-} 
g_k \leq   \frac{1}{\min_i \frac{p_i}{v_i}} \left( \frac{(\theta_k-\theta_{k+1})\sqrt{k+1}}{\alpha_0} + \frac{\alpha_0 \sum_{i=1}^n \frac{p_i L_i}{v_i^2}}{2 \sqrt{k+1}}\right).
\end{equation}

We have that the sequence $\left\{f(x_k)\right\}_{k\ge 0}$ is monotonically decreasing and  $f$ is bounded from below by $f(x_*)$, hence $f(x_*) \leq \theta_{k+1}\leq \theta_k\leq f(x_0)$ for all $k$. Letting $l=\lfloor K/2 \rfloor$, this implies that \[\sum_{j=l}^{2l} \theta_{j}-\theta_{j+1}  = \theta_l - \theta_{2l+1} \leq  f(x_0)-f(x_*)\eqdef C,\]
from which we conclude that there must exist $j\in \{l,\dots,2l\}$ such that $\theta_j-\theta_{j+1}\leq C/(l+1)$. This implies that
\begin{eqnarray*}
g_j &\overset{\eqref{eq:s9jd7d76d-}}{\leq} & \frac{1}{\min_i \frac{p_v}{v_i}} \left( \frac{(\theta_j-\theta_{j+1})\sqrt{j+1}}{\alpha_0} + \frac{\alpha_0 \sum_{i=1}^n \frac{p_i L_i}{v_i^2}}{2 \sqrt{j+1}}\right)\\
&\leq &  \frac{1}{\min_i \frac{p_v}{v_i}} \left( \frac{C\sqrt{j+1}}{\alpha_0 (l+1)} + \frac{\alpha_0 \sum_{i=1}^n \frac{p_i L_i}{v_i^2}}{2 \sqrt{j+1}}\right)\\
&\leq &  \frac{1}{\min_i \frac{p_v}{v_i}} \left( \frac{C\sqrt{2l+1}}{\alpha_0 (l+1)} + \frac{\alpha_0 \sum_{i=1}^n \frac{p_i L_i}{v_i^2}}{2 \sqrt{l+1}}\right)\\
&\leq & \frac{1}{\min_i \frac{p_v}{v_i} \sqrt{l+1}} \left( \frac{\sqrt{2}C}{\alpha_0 } + \frac{\alpha_0 \sum_{i=1}^n \frac{p_i L_i}{v_i^2}}{2}\right)\\
&\leq & \frac{1}{\min_i \frac{p_v}{v_i} \sqrt{K/2}} \left( \frac{\sqrt{2}C}{\alpha_0 } + \frac{\alpha_0 \sum_{i=1}^n \frac{p_i L_i}{v_i^2}}{2}\right)\\
&\overset{\eqref{eq:isjss8sus-}}{\leq} & \epsilon.
\end{eqnarray*}
\end{proof}

\begin{thm} Let Assumption \ref{ass:M-smooth} be satisfied. Choose $\alpha_{i_k}= \frac{\epsilon}{n v_{{i_k}}}$ where $\sum_{i=1}^n \frac{p_i L_i}{v_i^2} < 2n \left(\min_i \frac{p_i}{v_i}\right)$. If
\begin{equation} \label{eq:isjss8sus-2}
   K \ge    \frac{2 n \left(f(x_0)-f(x_*)\right)}{\min_i \frac{p_i}{v_i} \left(1 -\frac{\sum_{i=1}^n  \frac{p_i {L_i}}{{v_{i}^2}}}{2n \min_i \frac{p_i}{v_i}} \right)  \epsilon^2},
\end{equation}
then  $\displaystyle \min_{k=0,1,\dots,K} \mathbb{E} \left[ \|\nabla f(x_k)\|_{1} \right] \leq \epsilon$.
\end{thm}
\begin{proof}  
From equation (\ref{eq:condexp}) we have 
\begin{eqnarray}\label{eq:condexp-}
\mathbb{E} \left[ f(x_{k+1}) \;|\; x_k \right] &\leq& f(x_k)-  \underbrace{\mathbb{E} \left[ \alpha_{i_k} |\ve{\nabla f(x_k)}{e_{i_k}}| \;|\; x_k  \right]}_{\circledOne} + \underbrace{\frac{1}{2} \mathbb{E} \left[ \alpha_{i_k}^2 L_{i_k} \right]}_{\circledTwo}.
\end{eqnarray}

From \circledOne we have 
\begin{eqnarray*}
\mathbb{E} \left[ \alpha_{i_k} |\ve{\nabla f(x_k)}{e_{i_k}}| \;|\; x_k \right] &=& \frac{\epsilon}{n} \sum_{i=1}^n \frac{p_i}{{v_{i}}} |\ve{\nabla f(x_k)}{e_{i_k}}| \\
&\ge& \frac{\epsilon \min_i \frac{p_i}{v_i}}{ n} \|\nabla f(x_k)\|_1.
\end{eqnarray*}

From \circledTwo we have
\begin{eqnarray*}
\frac{1}{2} \mathbb{E} \left[ \alpha_{i_k}^2 L_{i_k} \right] &= &\frac{1}{2} \frac{\epsilon^2}{n^2} \sum_{i=1}^n \frac{p_i {L_i}}{{v_{i}^2}} 
\end{eqnarray*}
By injecting \circledOne and \circledTwo in (\ref{eq:condexp-}) and taking the expectation we get 
\begin{equation} 
\theta_{k+1}  \leq \theta_{k} - \frac{\epsilon}{n} \min_i \frac{p_i}{v_i} g_k  + \frac{\epsilon^2}{2 n^2}\sum_{i=1}^n \frac{p_i {L_i}}{{v_{i}^2}},
\end{equation}
where $\theta_{k} = \mathbb{E} \left[ f(x_k) \right]$  and $g_k = \mathbb{E} \left[ \|\nabla f(x_k)\|_1 \right]$. By re-arranging the terms we get:

\begin{equation}
\label{eq:s9jd7d76d} 
g_k \leq    \left( \frac{(\theta_k-\theta_{k+1})n }{\epsilon \min_i \frac{p_i}{v_i}} + \frac{\epsilon \sum_{i=1}^n \frac{p_i {L_i}}{{v_{i}^2}}}{2 n \min_i \frac{p_i}{v_i}}\right).
\end{equation}
We have that the sequence $\left\{f(x_k)\right\}_{k\ge 0}$ is monotonically decreasing and  $f(x)$ is bounded from below by $f(x_*)$, hence $f(x_*)\leq \theta_{k+1}\leq \theta_k\leq f(x_0)$ for all $k$. Letting $l=\lfloor K/2 \rfloor$, this implies that \[\sum_{j=l}^{2l} \theta_{j}-\theta_{j+1}  = \theta_l - \theta_{2l+1} \leq  f(x_0)-f(x_*) = C,\]
from which we conclude that there must exist $j\in \{l,\dots,2l\}$ such that $\theta_j-\theta_{j+1}\leq C/(l+1)$.
This implies that
\begin{eqnarray*}
g_k &\leq&    \left( \frac{C n }{\epsilon \min_i \frac{p_i}{v_i} (l+1)} + \frac{\epsilon \sum_{i=1}^n \frac{p_i {L_i}}{{v_{i}^2}}}{2 n \min_i \frac{p_i}{v_i}}\right)\\
&\leq& \left( \frac{C n }{\epsilon \min_i \frac{p_i}{v_i} K/2} + \frac{\epsilon \sum_{i=1}^n \frac{p_i {L_i}}{{v_{i}^2}}}{2 n \min_i \frac{p_i}{v_i}}\right)\\
&\overset{\eqref{eq:isjss8sus-2}}{\leq}&  \epsilon.
\end{eqnarray*}
\end{proof}

\section{Convex case}
We state a lemma which will be useful latter on in the analysis
\begin{lemma}
Let assumption \ref{ass:M-smooth} be satisfied. Let  $\alpha_{i_k} = \frac{|f(x_k + t e_{i_k}) - f(x_k)|}{tv_{i_k}}$, then the following inequality holds:
\begin{equation}
\begin{aligned}
f(x_{k+1}) &\leq f(x_k) - \frac{|\langle \nabla f(x_k),e_{i_k} \rangle|^2}{v_{i_k}} + \frac{t}{2 v_{i_k}} |\langle \nabla f(x_k) ,e_{i_k} \rangle| L_{i_k} \\
& + \frac{1}{2 v_{i_k}^2} |\langle \nabla f(x_k),e_{i_k} \rangle|^2 L_{i_k} + \frac{t}{2 v_i^2}  |\langle \nabla f(x_k),e_{i_k} \rangle| L_{i_k}^2  + \frac{t^2}{8v_{i_k}^2} L_{i_k}^3
\label{lemma1}
\end{aligned}
\end{equation}
\end{lemma}
\begin{proof}
It follows directly by noting that 
\begin{align*}
    \frac{|\langle\nabla f(x_k),e_{i_k}\rangle|}{v_{i_k}} - \frac{1}{2} L_{i_k} \leq \frac{|f(x_k+ t e_{i_k}) - f(x_k)|}{t v_{i_k}}\leq \frac{|\langle\nabla f(x_k),e_{i_k}\rangle|}{v_{i_k}} + \frac{1}{2} L_{i_k}.
\end{align*}
which follows in a similar fashion to Theorem [13] in \cite{Bergou_2018}.
\end{proof}

\begin{thm} 
Let Assumptions~\ref{ass:M-smooth} and \ref{ass:level_sets} be satisfied. Choose  $\alpha_{i_k} = \frac{\alpha_0\left(f(x_k) - f(x_*)\right)}{v_{i_k}}$, where $0<\alpha_0 <\frac{2 \min_i \frac{p_i}{v_i}}{R_0 \sum_{i=1}^n \frac{p_i L_i}{v_i^2}}$.  
 If 
 \begin{equation}
 \begin{aligned}
     k \ge \frac{1}{\frac{\alpha_0}{R_0} \min_i \frac{p_i}{v_i} - \frac{\alpha_0^2}{2} \sum_{i=1}^n \frac{p_i L_i}{v_i^2}}\left(\frac{1}{\epsilon} - \frac{1}{r_0}\right),
 \end{aligned}
\end{equation}
then $ \mathbb{E} \left[f(x_k)\right] - f(x_*)   \le \epsilon$.
\end{thm}
\begin{proof}
We have from Lemma \ref{lemma:key_lem}
  \begin{align}  
  \label{eq:main-res}
    \mathbb{E}\left[f(x_{k+1})\;|\; x_k\right] \leq f(x_k) -  \underbrace{\mathbb{E} \left[\alpha_{i_k} | \langle \nabla f(x_k), e_{i_k} \rangle| \;|\; x_k\right]}_{\circledOne} + \underbrace{\frac{1}{2} \mathbb{E} \left[\alpha_{i_k}^2 L_{i_k} \;|\; x_k\right]}_{\circledTwo}.
  \end{align}
  From \circledOne we have 
  \begin{align*}
   \mathbb{E} \left[\alpha_{i_k} | \langle \nabla f(x_k), e_{i_k} \rangle| |x_k\right] &= \alpha_0\left(f(x_k) - f(x_*)\right)\sum_{i=1}^n \frac{p_i}{v_i} | \langle \nabla f(x_k), e_i \rangle| \\
   &\ge \alpha_0 \left(f(x_k) - f(x_*)\right) \min_i \frac{p_i}{v_i} \|\nabla f(x_k) \|_1 \\
   &\stackrel{\eqref{eq:R_0_condition}}{\ge} \frac{\alpha_0}{R_0} \left(f(x_k) - f(x_*)\right)^2 \min_i \frac{p_i}{v_i}
     \end{align*}
     
  From \circledTwo we have 
  \begin{align*}
  \frac{1}{2}   \mathbb{E} \left[\alpha_{i_k}^2 L_{i_k} | x_k\right] = \frac{\alpha_0^2 \left(f(x_k) - f(x_*)\right)^2}{2} \sum_{i=1}^n \frac{p_i L_i}{v_i^2}.
  \end{align*}
  
Let $r_k = \mathbb{E} \left[f(x_k)\right] - f(x_*)$. By substituting \circledOne and \circledTwo in  (\ref{eq:main-res}) and then taking the expectation we get
  \begin{align*}
      r_{k+1} \le r_k - \frac{\alpha_0 r_k^2}{R_0} \min_i \frac{p_i}{v_i} + \frac{\alpha_0^2 r_k^2}{2} \sum_{i=1}^n \frac{p_i L_i}{v_i^2}  = r_k - a r_k^2,
  \end{align*}
 where $a = \frac{\alpha_0}{R_0} \min_i \frac{p_i}{v_i} - \frac{\alpha_0^2}{2} \sum_{i=1}^n \frac{p_i L_i}{v_i^2}.$ 
Therefore,
\begin{align*}
\frac1{r_{k+1}} 
-\frac1{r_k} =\frac{r_k - r_{k+1}}{r_k r_{k+1} }
 \geq 
 \frac{r_k - r_{k+1}}{r_k^2 } 
{\geq} a.
\end{align*} 
Therefore, we have $\frac{1}{r_k} \geq  \frac{1}{r_0} + ka $
and hence $r_k \le \frac{1}{\frac{1}{r_0} + ka}$. Therefore, if $k \ge \frac{1}{a}\left(\frac{1}{\epsilon} - \frac{1}{r_0}\right),$ then 
$r_k  \le  \frac{1}{\frac{1}{r_0} + ka} \le  \epsilon$.
\end{proof}

Note that the stepsizes in the previous theorem depend on the optimal value $f(x_*)$. In practice, we cannot always use these stepsizes as we usually do not know $f(x_*)$. Next theorem will suggest stepsizes that are independent from $f(x_*)$ for which we get an optimized complexity as well.

\begin{thm} 
Let Assumptions \ref{ass:M-smooth} and \ref{ass:level_sets} be satisfied. Choose $\alpha_{i_k} = \frac{|f({x}_k + t e{i_k}) - f({x}_k)|}{tv_{i_k}}$, then for a small t that satisfies

\begin{equation}
\begin{aligned}
t \leq &\min \left\{\frac{\epsilon^2 \min_i \frac{p_i}{v_i} }{8 R_0^2 n \max_i \sqrt{p_i} \sqrt{2 n{L}r_0}}, \frac{\epsilon}{R_0}\sqrt{\frac{\min_i\frac{p_i}{v_i}}{n \sum_{i=1}^n \frac{p_i L_i^3}{v_i^2}}}, \frac{\epsilon}{2 \max_i \sqrt{p_i} \sqrt{2 n{L} r_0}}, 2 \sqrt{\frac{\epsilon}{\sum_{i=1}^n \frac{p_i L_i^3}{v_i^2}}} \right\},
\label{cvx_range_t}
\end{aligned}
\end{equation}
 where $r_0 = f(x_0) - f(x_*),$ 
 and if 
\begin{equation}
\begin{aligned}
    k \ge  \frac{8R_0^2 n}{\min_i \frac{p_i}{v_i}} \left(\frac{1}{\epsilon} - \frac{1}{r_0} \right), 
\end{aligned}
\end{equation}
then $r_k = \mathbb{E}[f({x}_k)] - f(x_*) \leq \epsilon$.
\end{thm}
\begin{proof}
Taking the expectation of \eqref{lemma1} on $x_k$ with the choice of $\alpha_{i_k}$ we have:
\begin{equation}
\begin{aligned}
    & \mathbb{E}\left[f(x_{k+1}) | x_k\right] \leq f(x_k) + \underbrace{t \mathbb{E}\left[|\langle f(x_k),e_{i_k}\rangle| | x_k\right]}_{\circledOne} - \frac{1}{2} \underbrace{\mathbb{E} \left[\frac{1}{v_{i_k}} |\langle \nabla f(x_k),e_{i_k}\rangle|^2 | x_k\right]}_{\circledTwo} + \frac{t^2}{8} \sum_{i=1}^n \frac{p_i L_i^3}{v_i^2}.
    \label{cvx_optimal_step_sz_main}
\end{aligned}
\end{equation}

As for \circledOne, taking the total expectation we have
\begin{eqnarray*}
\mathbb{E}\left[\mathbb{E}[|\langle\nabla f(x_k) , e_{i_k}\rangle| |x_k]\right] &=& \mathbb{E} \left[| \langle \nabla f(x_k),e_{i_k} \rangle |\right]  = \mathbb{E} \left[\sqrt{\left(\sum_{i=1}^n p_i |\nabla_i f(x_k)|\right)^2}\right] \\
&\leq &
\max_i \sqrt{p_i} \mathbb{E} \left[ \|\nabla f(x_k) \|_2 \right] \\ 
&\leq & \max_i \sqrt{p_i} \sqrt{\mathbb{E} \left[\| \nabla f(x_k)\|_2^2 \right]}\\
& \leq& \max_i \sqrt{p_i} \sqrt{2 n {L} \left( \mathbb{E}\left[f(x_k)\right] - f(x_*)\right)} \leq \max_i \sqrt{p_i} \sqrt{2 n {L} r_0},
\end{eqnarray*}
where  
 $r_0 = \mathbb{E}\left[f(x_0)\right] - f(x_*)$. Note that the first inequality follows by Jensen's inequality. The last inequality follows from the fact that $f$ 
 has Lipschitz gradient, with Lispchitz constant $ n L$ (this is a direct property from Assumption~\ref{ass:M-smooth}). 
 
  As for \circledTwo, taking total expectation, we have
\begin{equation}
\begin{aligned}
\mathbb{E}\left[\mathbb{E} \left[\frac{1}{v_{i_k}} |\langle \nabla f(x_k),e_{i_k}\rangle|^2 | x_k\right]\right] &= \mathbb{E}\left[\sum_{i=1}^n \frac{p_i}{v_i} |\nabla_i f(x_k)|^2 | x_k \right]\\
&\ge \mathbb{E}\left[\min_i \frac{p_i}{v_i} \|\nabla f(x_k)\|_2^2 | x_k \right]\\
&\ge \mathbb{E}\left[\frac{1}{n}  \min_i \frac{p_i}{v_i} \|\nabla f(x_k)\|_1^2 | x_k\right]\\
&\stackrel{\eqref{eq:R_0_condition}}{\ge} \frac{1}{n}  \min_i \frac{p_i}{v_i} \frac{1}{R_0^2} \mathbb{E} \left[\left(f(x_k) - f(x_*)\right)^2\right] = \frac{r_k^2}{R_0^2 n} \min_i \frac{p_i}{v_i}.
\end{aligned}
\end{equation}
where $r_k = \mathbb{E}\left[f(x_k) - f(x_*)\right]$ Note that the second inequality follows since for $h \in \R^n$ the following holds $\|h\|_2 \ge \frac{1}{\sqrt{n}}\|h\|_1$. Lastly, by subtracting $f(x_*)$ and taking expectation of \eqref{cvx_optimal_step_sz_main}, we have
\begin{align*}
r_{k+1} \leq r_k -  \underbrace{\frac{1}{2} \frac{r_k^2 \min_i \frac{p_i}{v_i}}{R_0^2 n }}_{c_1} + \underbrace{t \max_i \sqrt{p_i} \sqrt{2 n {L} r_0}}_{c_2} + \underbrace{\frac{t^2}{8} \sum_{i=1}^n \frac{p_i L_i^3} {v_i^2}}_{c_3}.
\end{align*}

Thus, we have that

\begin{equation}
\begin{aligned}
\frac{1}{r_{k+1}} - \frac{1}{r_k} &\ge \frac{1}{r_k\left(1 - c_1 r_k \right) + c_2 + c_3} - \frac{1}{r_k}  = \frac{r_k - r_k\left(1 - c_1 r_k \right) - c_2 - c_3}{r_k \left[r_k\left(1 - c_1 r_k\right) +c_2 +c_3\right]}
\label{cvx_long_expression}
\end{aligned}
\end{equation}

Note that by setting $t$ to be the the first two upper bounds in \eqref{cvx_range_t}, the numerator of \eqref{cvx_long_expression} is lower bounded as 
\begin{align*}
&c_1r_k^2 - t\left(\max_i \sqrt{p_i} \sqrt{2 n {L} r_0} + \frac{t}{8}\sum_{i=1}^n \frac{p_i L_i^3}{v_i^2} \right) \ge c_1r_k^2 - \frac{\epsilon^2 \min_i \frac{p_i}{v_i}}{4 R_0^2 n} \ge \frac{\min_i \frac{p_i}{v_i} r_k^2}{4 R_0^2 n}.
\end{align*}
Moreover, by setting $t$ to the last two upper bounds in \eqref{cvx_range_t}, the denominator of \eqref{cvx_long_expression} is lower bounded as 
\begin{align*}
&r_k \left[r_k\left(1 - c_1 r_k\right) +c_2 +c_3\right] \leq r_k^2 + (c_2 + c_3)r_k  \leq r_k^2 + \left(\frac{\epsilon}{2} + \frac{\epsilon}{2}\right)r_k \leq 2 r_k^2.  
\end{align*}

Then $\frac{1}{r_{k+1}} - \frac{1}{r_k} \ge \frac{\min_i \frac{p_i}{v_i}}{8R_0^2 n} = a$ then $ \frac{1}{r_k} \ge \frac{1}{r_0} + ka$. Hence if
\begin{align*}
    k \ge \frac{1}{a} \left(\frac{1}{\epsilon} - \frac{1}{r_0} \right) =  \frac{8 R_0^2 n}{\min_i \frac{p_i}{v_i}} \left(\frac{1}{\epsilon} - \frac{1}{r_0} \right)
\end{align*}
we have $\mathbb{E} \left[f(x_k) - f(x_*)\right] \leq \epsilon$.
\end{proof}

\section{Strongly Convex Case}
\begin{thm} 
Let Assumptions~\ref{ass:M-smooth} and \ref{ass:f-SC} be satisfied. Choose $\alpha_{i_k} = \frac{\alpha_0}{v_{{i_k}}} \sqrt{2 \mu \left(f(x_k) - f(x_*) \right)}$, where $0 < \alpha_0 < \frac{2}{\sum_{i=1}^n \frac{p_i L_i}{v_i^2}}$. If
\begin{equation}
    \begin{aligned}
    k \ge \frac{\max_i \frac{v_i^2}{p_i^2}}{\lambda} \frac{1}{2\alpha_0 - \alpha_0^2 \sum_{i=1}^n \frac{p_i L_i}{v_i^2}} \log \left(\frac{f(x_0) - f(x_*)}{\epsilon} \right),
    \end{aligned}
\end{equation}
then $\mathbb{E} \left[f(x_k) - f(x_*)\right] \leq \epsilon$. 
\end{thm}
\begin{proof}
Since $f$ is $\lambda$-strongly convex, then:
\begin{equation}
\begin{aligned}
f(x_{k+1}) &\ge f(x_{k}) + \langle \nabla f(x_k), h \rangle + \frac{\lambda}{2} \| h \|_2^2 \\
&\stackrel{\mu = \frac{\lambda}{\max_i \frac{v_{i}^2}{p_i^2}}}{\ge} f(x_{k}) + \langle \nabla f(x_k), h \rangle + \frac{\mu}{2} \| h \|_{p^{-2} \circ v^2}^2, \\
\end{aligned}
\end{equation}
 where $\| h \|_{p^{-2} \circ v^2}^2 = \sum_{i=1}^n \frac{v_i^2}{p_i^2} h_i^2$. 
Minimizing the right hand side in $h$ and substituting again, we have the inequality:
\begin{equation}
\begin{aligned}
\| \nabla f(x_k)\|_{p^2 \circ v^{-2}} \eqdef \sqrt{\sum_{i=1}^n \frac{p_i^2}{v_i^2} \nabla_i f(x_k)^2} &\ge \sqrt{2 \mu \left(f(x_k) - f(x_*)\right)}
\label{strong_conv_ineq}
\end{aligned}
\end{equation}
By subtracting $f(x_*)$ from both sides of Lemma \ref{lemma:key_lem} we get
\begin{equation}
\begin{aligned}
&\mathbb{E}\left[f(x_{k+1})|x_k\right] - f(x_*) \leq f(x_k) -f(x_*)-\underbrace{ \mathbb{E} \left[\alpha_{i_k}  | \langle\nabla f(x_k),e_{i_k} \rangle | | x_k \right]}_{\circledOne} + \underbrace{\mathbb{E} \left[\frac{\alpha_{i_k}^2}{2} L_{i_k} | x_k\right]}_{\circledTwo}
\label{smoothness_and_rds}
\end{aligned}
\end{equation}
From \circledOne, we have 
\begin{equation}
\begin{aligned}
     \mathbb{E} \left[\alpha_{i_k} | \langle\nabla f(x_k),e_{i_k}\rangle | | x_k\right] &= 
     \alpha_0 \sqrt{2 \mu \left(f(x_k) - f(x_*)\right)} \sum_{i=1}^n \frac{p_i}{v_i} |\nabla_i f(x_k)|.
    \label{total_exp}
\end{aligned}
\end{equation}
Since the following holds:
\begin{align*}
    \sum_{i=1}^n \frac{p_i}{v_i} |\nabla_i f(x_k)| = \sqrt{\left(\sum_{i=1}^n \frac{p_i}{v_i} |\nabla_i f(x_k)|\right)^2} \ge \sqrt{\sum_{i=1}^n \frac{p_i^2}{v_i^2} | \nabla_i f(x_k)|^2} = \|\nabla f(x_k)\|_{p^2 \circ v^{-2}},
\end{align*}
\noindent Thus, combining the previous result with \eqref{total_exp}, we have
\begin{align*}
    \mathbb{E} \left[\alpha_{i_k} | \langle\nabla f(x_k),e_{i_k}\rangle | | x_k\right]  &\ge \alpha_0 \sqrt{2 \mu \left(f(x_k) - f(x_*)\right)} \|\nabla f(x_k) \|_{p^2 \circ v^{-2}} \\
    & \stackrel{\eqref{strong_conv_ineq}}{\ge} 2 \mu \alpha_0  \left(f(x_k) - f(x_*)\right)
\end{align*}
From \circledTwo, we have
\begin{equation}
\begin{aligned}
   \mathbb{E}\left[\frac{\alpha_{i_k}^2}{2} L_{i_k} \;|\; x_k \right] 
    &= \mu \alpha_0^2 \left(f(x_k) - f(x_*)\right) \sum_{i=1}^n \frac{p_i L_i}{v_i^2}
    \label{second_term_strong_cov}
\end{aligned}
\end{equation}

Lastly, substituting the results of \circledOne and \circledTwo in \eqref{smoothness_and_rds} and taking the expectation with respect to $x_k$ where $r_k = \mathbb{E}[f(x_k) - f(x_*)]$, we get

\begin{equation}
\begin{aligned}
r_{k+1} & \leq r_{k}  - 2 \alpha_0 \mu r_k + \mu \alpha_0^2 r_k \sum_{i=1}^n \frac{p_i L_i}{v_i^2} \\
&= \left(1 -  \frac{\lambda}{\max_i \frac{v_i}{p_i}}\left(2 \alpha_0 - \alpha_0^2 \sum_{i=1}^n \frac{p_i L_i}{v_i^2} \right)\right)r_{k}.   \\
\end{aligned}
\end{equation}

\begin{equation}
    \begin{aligned}
    K \ge \frac{\max_i \frac{v_i^2}{p_i^2}}{\lambda} \frac{1}{2 \alpha_0 - \alpha_0^2 \sum_{i=1}^n \frac{p_i L_i}{v_i^2}} \log \left(\frac{f(x_0) - f(x_*)}{\epsilon}\right),
    \end{aligned}
\end{equation}

then $\mathbb{E} \left[f(x_k) - f(x_*)\right] \leq \epsilon$.
\end{proof}

\begin{thm} 
Let Assumptions \ref{ass:M-smooth} and \ref{ass:f-SC} be satisfied. Choose $\alpha_{i_k} = \frac{|f(x_k + t e_{i_k}) - f(x_k)|}{t v_{i_k}}$ for some small t, $\mu = \frac{\lambda}{\max_i\frac{v_i}{p_i^2}}$. If 
\begin{equation}
\begin{aligned}
k \ge \frac{\max_i \frac{v_i}{p_i^2}}{\lambda} \log \left(\frac{2\left(f(x_0) - f(x_*)\right)}{\epsilon} \right)
\end{aligned}
\end{equation}
then $\mathbb{E}\left[ f(x_k) - f(x_*)\right] \leq \epsilon$.
\end{thm}
\begin{proof}
Taking the expectation of \eqref{lemma1} conditioned on $x_k$, we have:
\begin{equation}
\begin{aligned}
    & \mathbb{E}\left[f(x_{k+1}) | x_k\right] \leq f(x_k) + \underbrace{t \mathbb{E}\left[|\langle f(x_k),e_{i_k}\rangle| | x_k\right]}_{\circledOne} - \frac{1}{2} \underbrace{\mathbb{E} \left[\frac{1}{v_k} |\langle \nabla f(x_k),e_{i_k}\rangle|^2 | x_k\right]}_{\circledTwo} + \frac{t^2}{8}\sum_{i=1}^n \frac{p_i L_i^3}{v_i^2}.
    \label{strong_covx_optimal_step}
\end{aligned}
\end{equation}

\noindent Taking total expectation in \circledOne we have 
\begin{eqnarray*}
\mathbb{E}\left[\mathbb{E}[|\langle\nabla f(x_k) , e_{i_k}\rangle| |x_k]\right] &=& \mathbb{E} \left[| \langle \nabla f(x_k),e_{i_k} \rangle |\right]  = \mathbb{E} \left[\sqrt{\left(\sum_{i=1}^n p_i |\nabla_i f(x_k)|\right)^2}\right] \\
&\leq &
\max_i \sqrt{p_i} \mathbb{E} \left[ \|\nabla f(x_k) \|_2 \right] \\ 
&\leq & \max_i \sqrt{p_i} \sqrt{\mathbb{E} \left[\| \nabla f(x_k)\|_2^2 \right]}\\
& \leq& \max_i \sqrt{p_i} \sqrt{2 n {L} \left( \mathbb{E}\left[f(x_k)\right] - f(x_*)\right)} \leq \max_i \sqrt{p_i} \sqrt{2 n {L} r_0},
\end{eqnarray*}
where  $r_0 = \mathbb{E}\left[f(x_0)\right] - f(x_*)$. Note that the first inequality follows by Jensen's inequality. The last inequality follows from the fact that $f$ 
 has Lipschitz gradient, with Lispchitz constant $ n L$. 
As for \circledTwo, note that since  $\|h\|_{v \circ p^{-1}}^2 \leq \max_i \frac{v_i}{p_i} \|h\|_2^2$; thus, by strong convexity we have that 
\begin{equation}
\begin{aligned}
f(x_{k+1}) &\stackrel{\mu = \frac{\lambda}{\max_i \frac{v_i}{p_i}}}{\ge} f(x_{k}) + \langle \nabla f(x_k), h \rangle + \frac{\mu}{2} \| h \|_{p^{-1} \circ v}^2 \ge f(x_k) - \frac{1}{2\mu} \|\nabla f(x_k)\|^2_{p \circ v^{-1}},
\end{aligned}
\end{equation}
where $\| h \|_{p^{-1} \circ v}^2 = \sum_{i=1}^n \frac{v_i}{p_i} h_i^2$ and $\| \nabla f(x_k)\|_{p \circ v^{-1}} \eqdef \sqrt{\sum_{i=1}^n \frac{p_i}{v_i} \nabla_i f(x_k)^2}$. 
Thus, it follows that $\|\nabla f(x_k)\|^2_{p\circ v^{-1}} \ge 2 \mu (f(x_k) - f(x_*))$. Thereafter, taking the expectation of  \circledTwo and with the use of the total expectation rule,  \circledTwo can be lower bounded as follows
\begin{equation}
\begin{aligned}
\mathbb{E} \left[\left(\frac{|\langle \nabla f(x_k),e_{i_k} \rangle |}{\sqrt{v_{i_k}}}\right)^2 \right] & = \left(\sum_{i=1}^n \frac{p_i}{v_i} \nabla_i f(x_k)^2\right) = \|\nabla f(x_k)\|^2_{p \circ v^{-1}} \\
 & \stackrel{\mu = \frac{\lambda}{\text{max}_i \frac{v_i}{p_i}}}{\ge} 2\mu \left(\mathbb{E}[f(x_k)] - f(x_*)\right) = 2\mu r_k.
\end{aligned}
\end{equation}

\noindent By subtracting $f(x_*)$ in \eqref{strong_covx_optimal_step} and taking expectation, we get
\begin{equation}
\begin{aligned}
r_{k} &\leq r_{k-1} + t \max_i \sqrt{p_i} \sqrt{2 n{L} r_0} -\mu r_{k-1} + \frac{t^2}{8}\sum_{i=1}^n \frac{p_i L_i^3}{v_i^2} \\
&\leq (1 - \mu)^k r_{0} + \left( t \max_i \sqrt{p_i} \sqrt{2 n {L} r_0} + \frac{t^2}{8} \sum_{i=1}^n \frac{p_i L_i^3}{v_i^2} \right) \sum_i^{k-1} (1-\mu)^i\\
&\leq (1 - \mu)^k r_{0} + \left( t \max_i \sqrt{p_i} \sqrt{2 n {L} r_0} + \frac{t^2}{8} \sum_{i=1}^n \frac{p_i L_i^3}{v_i^2}\right) \frac{1}{\mu}\\
\end{aligned}
\end{equation}

Thus, by choosing $t$ such that $\left( t \max_i \sqrt{p_i} \sqrt{2 n {L} r_0} + \frac{t^2}{8} \sum_{i=1}^n \frac{p_i L_i^3}{v_i^2} \right)\frac{1}{\mu}  \leq \frac{\epsilon}{2}$ and if
\begin{align*}
    k \ge \frac{\max_i \frac{v_i}{p_i}}{\lambda} \log \left(\frac{2\left(f(x_0) - f(x_*)\right)}{\epsilon} \right),
\end{align*}
we have $\mathbb{E}\left[ f(x_k)\right]  - f(x_*)\leq \epsilon$. 
\end{proof}

\section{Estimating the Lipschitz Smoothness Constant}
As we explain in the Section~\ref{sec:exper}, we do not have an access to the Lipschitz smoothness constants of the function for the continuous control case. Instead, we estimate them using the points which have been evaluated. Our experiments suggest that, estimating the function directly using a parametric family works better than estimating smoothness constants directly. In other words, we estimate the function to be optimized directly using a parametric family ($\hat{f}(\cdot, \theta)$) and use its Lipschitz smoothness constants as an estimate. 

Consider the DFO step $n$; using the queried sampled $\{x_i, f(x_i)\}_{i \in [n-1]}$, we can estimate the function of interest by solving empirical risk minimization problem as: 
\begin{equation}
    \theta_n = \arg\min_\theta \sum_i  (f(x_i) - \hat{f}(x_i, \theta))^2
\end{equation}

We use the resulting $\hat{f}(\cdot, \theta_n)$ to compute Lipschitz smoothness constants and importance sampling weights. The parametric family we consider is a multi-layer perceptron with single hidden layer and tanh non-linearity. Input dimension is the policy size, output dimension is 1 and hidden layer dimension is chosen as 16 for \texttt{Swimmer-v1} and \texttt{Hopper-v1}, 64 for \texttt{HalfCheetah-v1} and \texttt{Ant-v1}, and 256 for \texttt{Humanoid-v1}. Learning has been performed using ADAM \cite{Kingma_2014} optimizer at each iteration. Step size (learning rate) has been chosen as $10^{-3}$ for all experiments. At each iteration we choose a $32$ random samples in uniform from $x_0,\dots,x_{n-1}$ and use it as a batch.

\end{document}